\documentclass[11pt]{article}
\usepackage{graphicx,wrapfig}
\graphicspath{figures/}
\usepackage{flafter}
\usepackage{leftidx}
\usepackage{mathrsfs}
\usepackage{amssymb,amsmath}
\usepackage{bbding}
\usepackage{fancyhdr}
\usepackage{mathrsfs}
\usepackage{color}
\usepackage{stmaryrd}
\usepackage{amsfonts}
\usepackage{latexsym}
\usepackage{psfrag}
\usepackage{graphicx,subfigure}
\usepackage{fancyhdr,graphicx}
\usepackage{multicol}
\usepackage{dsfont}
\usepackage{bbm}
\usepackage{indentfirst}
\usepackage{booktabs}
\usepackage[center]{caption2}
\usepackage{cite}
\usepackage{epstopdf}
\usepackage{verbatim}
\usepackage[amsthm,amsmath,thmmarks]{ntheorem}
\usepackage{diagbox}
\usepackage{booktabs}
\usepackage [latin1]{inputenc}
\usepackage{multirow}
\usepackage{enumerate}
\usepackage{enumitem}
\usepackage{algpseudocode,algorithm,algorithmicx}
\usepackage{appendix}

\allowdisplaybreaks
\renewcommand{\leq}{\leqslant}

\date{}
\newtheorem{theorem}{Theorem}[section]
\newtheorem{lemma}{Lemma}[section]
\newtheorem{remark}{Remark}[section]
\newtheorem{example}{Example}[section]

\numberwithin{equation}{section}

\newcommand{\zd}{\,\mathrm{d}}
\newcommand{\diff}{\triangledown_{\tau}}
\newcommand{\myvec}[1]{\boldsymbol{#1}}
\newcommand{\abs}[1]{\left|#1\right|}
\newcommand{\absb}[1]{\big|#1\big|}

\newcommand{\bra}[1]{\left(#1\right)}
\newcommand{\brab}[1]{\big(#1\big)}
\newcommand{\braB}[1]{\Big(#1\Big)}
\newcommand{\brabb}[1]{\bigg(#1\bigg)}
\newcommand{\brat}[1]{(#1)}

\newcommand{\kbrabb}[1]{\bigg[#1\bigg]}

\newcommand{\myinner}[1]{\left\langle#1\right\rangle}
\newcommand{\myinnerb}[1]{\big\langle#1\big\rangle}
\newcommand{\myinnerB}[1]{\Big\langle#1\Big\rangle}
\newcommand{\mynorm}[1]{\left\|#1\right\|}
\newcommand{\mynormb}[1]{\big\|#1\big\|}
\newcommand{\mynormB}[1]{\Big\|#1\Big\|}

\begin{document}
	\title{Convergence analysis of variable steps BDF2 method for the space fractional Cahn-Hilliard model}
	\author{Xuan Zhao\thanks{Corresponding author. School of
			Mathematics, Southeast University, Nanjing 210096, P. R. China (xuanzhao11@seu.edu.cn).}
		\quad Zhongqin Xue \thanks{School of
			Mathematics, Southeast University, Nanjing 210096, P. R. China. (zqxue@seu.edu.cn).}}
	\maketitle
	\normalsize
	
	\begin{abstract}
	An implicit variable-step BDF2 scheme is established for solving the space fractional Cahn-Hilliard equation derived from a gradient flow in the negative order Sobolev space $H^{-\alpha}$, $\alpha\in(0,1)$. The Fourier pseudo-spectral method is applied for the spatial approximation. The space fractional Cahn-Hilliard model poses significant challenges in theoretical analysis for variable time-stepping algorithms compared to the classical model, primarily due to the introduction of the fractional Laplacian. This issue is settled by developing a general discrete H\"older inequality involving the  discretization of the fractional Laplacian. Subsequently, the unique solvability and the modified energy dissipation law are theoretically guaranteed. We further rigorously provided the convergence of the fully discrete scheme by utilizing the newly proved discrete Young-type convolution inequality to deal with the nonlinear term. Numerical examples with various interface widths and mobility are conducted to show the accuracy and the energy decay for different orders of the fractional Laplacian. In particular, the multiple time scales evolutions of the solution are captured by an adaptive time-stepping strategy in simulations.
		
%

\vskip5mm


\end{abstract}
	\section{Introduction}\setcounter{equation}{0}
	Phase separation is a universal process occurring in many materials systems, wherein an initially homogeneous mixed state decomposes into different phases. The dynamics of phase segregation process in a multi-component system is usually simulated by the Cahn-Hilliard model \cite{Cahn1958J.Chemi.Physi.,Frohoff2021Phys.Rev.E,Toth2016Phys.Rev.E,Golovin2001PRL}. Furthermore, variants and extensions of the Cahn-Hilliard equation are also increasingly used in tumor growth\cite{Cristini2009JMB}, image processing\cite{Bertozzi2007IEEE Trans.ImageProcess.},  biology\cite{Khain2008Phys.Rev.}, etc. The classical Cahn-Hilliard equation is given in the following form
	$$ \partial_t \Phi-\Delta\left(\epsilon^{2}(-\Delta) \Phi+F'(\Phi)\right)=0,$$
	where $\Phi$ is an order parameter, $\epsilon$ is an interface width parameter, and $F(\Phi)$ is a double-well potential, typically in the form of $F(\Phi)=\frac{1}{4}(\Phi^2-1)^2$, which has two minima: $\Phi=\pm1$, corresponding to the pure phases.

	The simulation of the Cahn-Hilliard equation occurs on a long time scale.  Meanwhile, the parameter $\epsilon$ is very small compared to the characteristic length of the laboratory scale, in which case the phase transition tends to be sharper. Hence, energy stable schemes with high order time accuracy play crucial roles in the numerical simulation\cite{Akrivis2019SIAM,Liu2021J.Comput.Appl.Math.,Cheng2016J.Sci.Comput.,Cheng2019J.Comput.Appl.Math.,Chen2019SIAM,Liao2021math.NA,Shen2010Discret.Contin.Dyn.Syst.,Li2017J.Sci.Comput.,Lidong2022MC,Feng2015SIAM}. Nonuniform grids are further selected with adaptive time-stepping techniques\cite{Wei2022CHSAV,Zhang2012CCP,Luo2016CCP,Gomez2011JSP,Liao2021math.NA} for the efficient simulation, where the time step is changing with the variation of the energy or the numerical solutions.
	
Nevertheless, the Cahn-Hilliard equation cannot be rigorously derived as a macroscopic limit of microscopic models for interacting particles. Noticing that, the spatial convolution term, in the original physical model\cite{Cahn1958J.Chemi.Physi.}, which describes long-range interactions among particle was replaced by the classical Laplacian in the Cahn-Hilliard equation. 
    Furthermore, the space fractional Cahn-Hilliard (FCH) equations with fractional Laplacian operators \cite{Bosch2015SIAM} and the integro-partial-differential nonlocal Cahn-Hilliard equations\cite{Du2018nonlocal} with nonlocal convolution kernel in some cases are selected in practical application scenarios. A detailed excursion about several definitions of $(-\Delta)^{\alpha}$ is displayed in \cite{Lischkea2020JCP}, besides, differences in the boundary behaviors of solutions to equations involving the fractional Laplacian posed with different definitions are identified numerically.
    The uniqueness and the oscillation results for radial solutions of linear and nonlinear equations that involve the fractional Laplacian are derived in arbitrary space dimension \cite{Frank2016CPAM}. The numerical computations using the spectral method, the finite element method and the finite difference method of the fractional Laplacian are discussed in\cite{Minden2020SIAM,Huang2014SIAM,Zhang2018SIAM,Faustmann2022SIAM,Borthagaray2021SIAM,Aceto2017SIAM,Tang2017SIAM,Xu2021JCAM}. The generalized fractional variant of the Cahn-Hilliard equation\cite{Goro2016J.Differ.Equ.} is written in the following form	
	\begin{align}\label{general FCH}
		\partial_t \Phi+(-\Delta)^{\alpha}\left(\epsilon^{2}(-\Delta)^{\beta} \Phi+F'(\Phi)\right)=0,
	\end{align}
	where $\alpha\in [0,1], \beta\in (0,1]$. Summarizing the achievements in the study of different types of the space FCH equation\cite{Weng2017AMM,Mark2017SIAM,Mark2017Chaos,Wang2019J.Comput.Appl.Math.,Zhao2021Bit,Bu2020Appl.Numer.Math}, the models are mainly classified into two cases:
	
	\textbf{Case I} $\left[\alpha\in(0,1),\beta=1\right]$: The space FCH equation is derived by the $H^{-\alpha}$ gradient flow with respect to the free energy functional $E_1[\Phi]$:
	\begin{align}\label{FCH-1 energy}
		E_1[\Phi]=\int_{\Omega}\left(\frac{\epsilon^{2}}{2}\left|\nabla \Phi\right|^{2}+F(\Phi)\right)\zd\mathbf{x}.
	\end{align}
	\indent\setlength{\parindent}{2em}\textbf{Case II} $\left[\alpha=1,\beta\in\left(\frac{1}{2},1\right)\right]$: The space FCH equation can be viewed as the $H^{-1}$ gradient flow associated with the non-local free energy functional $E_2[\Phi]$:
	\begin{align}\label{FCH-2 energy}
		E_2[\Phi]=\int_{\Omega}\left(\frac{\epsilon^{2}}{2}\left|(-\Delta)^{\frac{\beta}{2}} \Phi\right|^{2}+F(\Phi)\right)\zd\mathbf{x}.
	\end{align}
   Specially, when $\alpha=0$ and $\beta\in(\frac{1}{2},1)$, Eq.(\ref{general FCH}) corresponds to the space fractional Allen-Cahn equation \cite{Tang2017JSC,Song2016CMAME}, when $\alpha=\beta=1$, Eq.(\ref{general FCH}) turns out to be the classical Cahn-Hilliard equation and when $\alpha=0,\beta=1$, Eq.(\ref{general FCH}) is the classical Allen-Cahn equation.
		
	In this paper, we consider the space FCH equation of \textbf{Case I} in \cite{Mark2017SIAM}, which is rewritten as
	\begin{align}\label{cont: Problem-FCH}
		\partial_t \Phi=-\kappa(-\Delta)^{\alpha}\mu\quad\text{with}\quad\mu:=\frac{\delta E_1}{\delta\Phi}=\epsilon^{2}(-\Delta) \Phi+F'(\Phi),\quad (\mathbf{x},t)\in\Omega\times(0,T],
	\end{align}
	where $\alpha\in(0,1)$ is the order of the space fractional operator, $\epsilon$ is a positive constant and $\kappa$ is a mobility parameter. For simplicity, we set $\Omega=(0,L)^{2}$ and $F(\Phi)=\frac{1}{4}(\Phi^{2}-1)^{2}$. Assume that $\Phi(\cdot,t)$ is periodic over the domain $\Omega$ and satisfies $\Phi(\mathbf{x},0)=\Phi_{0}(\mathbf{x})$. The fractional Laplacian $(-\Delta)^{\alpha}$ is defined by the Fourier decomposition
	$$(-\Delta)^{\alpha}u=\sum_{m,n\in \mathbb{Z}}\lambda_{m,n}^{\alpha}\widehat{u}_{m,n}e^{i\nu(mx+ny)},$$ 
	where $\lambda_{m,n}:=\nu^2\bra{m^2+n^2}$, $\nu=2\pi/L$ and $\widehat{u}_{m,n}$ corresponds to the Fourier coefficients
	$$
	\widehat{u}_{m,n}=\frac{1}{L^{2}} \int_{\Omega} u(\mathbf{x}) e^{-i\nu(mx+ny)} \mathrm{d}\mathbf{x}.
	$$
	Moreover, the solution of the space FCH equation \eqref{cont: Problem-FCH} possesses the mass conservation
	$$\brab{\Phi(t),1}=\brab{\Phi(0),1},$$
	and the energy dissipation law
	\begin{align}\label{cont:energy dissipation}
		\frac{\zd{E}}{\zd{t}}=\left(\frac{\delta E_1}{\delta\Phi},\partial_t \Phi\right)=-\kappa\mynormb{(-\Delta)^{\frac{\alpha}{2}}\left(\epsilon^2(-\Delta)\Phi+F'(\Phi)\right)}^2\le 0.
	\end{align}
	
There are some existing works devoted to the theoretical analysis and numerical simulations of the space FCH equation. Akagi et.al.\cite{Goro2016J.Differ.Equ.} discussed a more general fractional variant of the Cahn-Hilliard equation with homogeneous Dirichlet boundary conditions of the solid type. The existence and uniqueness of the weak solutions along with some regularity properties and energy inequalities were provided. Significant singular limits, obtained as orders of the fractional Laplacian tended to 0 respectively, were analyzed rigorously. Lately, they also established the global existence of the weak solutions and the convergence of each solution to the equilibrium for the fractional Cahn-Hilliard equation\cite{AkagiJFA2019}. Ainsworth and Mao\cite{Mark2017SIAM} investigated the well-posedness of the space FCH equation. An $L^\infty$ bound of the numerical solution and the spectral convergence rate for the Fourier Galerkin method were proved in the semi-discrete sense. The stability of the fully discrete scheme was provided with the first-order semi-implicit method for the time discretization. Bu et al.\cite{Bu2020Appl.Numer.Math} constructed three  energy stable schemes based on the backward Euler method, the Crank-Nicolson method and the BDF2 scheme respectively for the space FCH equation with the help of adding stabilizing terms. The mass conservation, the unique solvability and the energy stability of the semi-discrete schemes were proved. The error estimates for the fully discrete schemes with Fourier spectral method were derived under the uniform time step sizes, by assuming that the nonlinear function was Lipschitz continuous.

The main purpose of this work is to develop an implicit variable-step BDF2 scheme for the space FCH equation and establish the corresponding theoretical analysis. A modified energy by adding a small term to the discrete version of the original free energy functional is constructed to justify the energy stability. Subsequently, a novel discrete $L^6$ norm estimate along with the Young-type convolution inequality lead to the $L^2$ norm error estimate. As is known to all, a full matrix is need to be solved at each time step due to the nonlocal property of the fractional Laplacian. Meanwhile, the dynamics of the space FCH model involves the multiple time scales behaviors that the free energy dissipates at different rates in different stages. In addition, the energy evolution requires more time to reach the steady state for smaller value of fractional order. Thus, based on the theoretical results, an adaptive time-stepping technique is adopted to improve the computational efficiency of long simulation and capture multiple time scales evolutions. 

The main contribution of the paper is the theoretical proof and the error estimate for the implicit variable-step BDF2 scheme for the space FCH equation. The key novelty is the discrete Young-type convolution inequality proved in Lemma \ref{lem:DOC quadr form Ha-H1+a embedding inequ}, which makes the analysis framework constructed in the current manuscript available for the higher order backward difference formula upon the requirements that the corresponding discrete orthogonal convolution (DOC) kernels are positive definite. Besides, it could be further extended to the general problems involving fractional Laplacian.
 The adaptive time-stepping strategy alleviates the huge CPU time consumption brought by the nonlocality of the fractional Laplacian, which is indicated by Table \ref{CPU time2}, especially in long time simulation. 
		
The remainder of this paper is structured as follows. In section 2, we construct an implicit variable-step BDF2 scheme for the space FCH equation \eqref{cont: Problem-FCH}. The mass conservation, the unique solvability and the energy stability of the proposed scheme are proved in section 3. In section 4, we carry out the detailed convergence analysis, where the discrete Young-type convolution inequality dealing with the fractional Laplacian is proved. Numerical experiments are provided to verify the accuracy and the effectiveness of the proposed scheme in section 5. 	
	
	\section{The implicit variable-step BDF2 scheme}
	In this section, an implicit variable-step BDF2 scheme for the space FCH equation is constructed, with the Fourier pseudo-spectral method for the spatial approximations.	
	\subsection{The spatial discretization}
	The Fourier pseudo-spectral spatial discretization of the space FCH equation for $\Omega=(0,L)^{2}$, subjected to the $L$-periodic boundary conditions, will be performed on the uniform spatial grids. For a positive even integer $M$, we set mesh size in space: $h_x=h_y=h:=L/M$. Then, the spatial grid points are denoted as  $\Omega_{h}:=\big\{\mathbf{x}_{h}=(x_i,y_j)\,|\,x_i=ih,y_j=jh,\,0\le i,j \le M-1\big\}$. Let $\mathbb{V}_{h}:=\{v_h=v(\mathbf{x}_{h})\,|\, \mathbf{x}_h\in\Omega_{h}\}$ be the space of grid functions defined on $\Omega_{h},$ for any grid functions $v,w\in\mathbb{V}_{h}$, the discrete inner product and discrete norms are given respectively as
	\begin{align*}
		\myinnerb{v,w}:=h^2\sum_{\mathbf{x}_h\in\Omega_{h}}v_h \bar{w}_h\,,\qquad \mynormb{v}:=\mynormb{v}_{l^{2}}=\sqrt{\myinnerb{v,v}}\,,\quad
		\mynormb{v}_{{{l}^{q}}}:=\sqrt[q]{h^{2} \sum_{x_{h} \in \Omega_{h}}\left|v_{h}\right|^{q}}.
	\end{align*}
    {In this paper, we take the following notations for the sake of brevity}
    $$\sum_{m,n}:=\sum_{m,n= -M/2}^{M/2-1},\qquad \sum_{(m,n)\neq0}:=\sum_{\mbox{\tiny$\begin{array}{c}
    				m,n=-M/2\\
    				\bra{m,n}\neq \mathbf{0}
    			\end{array}$}}^{M/2-1}.$$  
    We introduce space $S_{M}$ which consists of trigonometric polynomials of degree up to $M/2$.
	For a periodic function $v(\mathbf{x})\in L^{2}(\Omega)$, let $P_{M}: L^{2}(\Omega) \rightarrow S_{M}$ be the $L^{2}$ projection \cite{Shen2011Springer} onto $S_{M}$ as
	$$
	\bra{P_Mv}(\mathbf{x})=\sum_{m,n} \widehat{v}_{m, n} e^{\mathrm{i}\nu\bra{mx+ny}},
	$$
	where $\widehat{v}_{m,n}$ refer to the Fourier coefficients.
	We also denote the following interpolation operator $I_{M}: L^{2}(\Omega) \rightarrow S_{M}$
	$$
    \bra{I_{M}v}(\mathbf{x})=\sum_{m,n} \widetilde{v}_{m,n}e^{\mathrm{i}\nu\bra{mx+ny}},
	$$
	where the pseudo-spectral coefficients are computed based on the fact that $\bra{I_{M}v}(\mathbf{x_h})=v_h$. Thus,  the discrete first-order and the second-order derivatives given by
	\begin{align*}
		\mathcal{D}_xv_h:=\sum_{m,n}
		\bra{\mathrm{i}\nu m}\widetilde{v}_{m,n}
		e^{\mathrm{i}\nu\bra{mx_i+ny_j}}\,,\qquad\mathcal{D}_x^2v_h:=\sum_{m,n}
		\bra{\mathrm{i}\nu m}^2\widetilde{v}_{m,n}
		e^{\mathrm{i}\nu\bra{mx_i+ny_j}}\,.
	\end{align*}
	Similarly, we can define the differentiation  operators $\mathcal{D}_y$ and $\mathcal{D}_y^2$ .
	In turn, the discrete gradient $\nabla_h$ and Laplacian $\Delta_h$ in the point-wise sense is defined as
	
	\[
	\nabla_hv_h := \left(\mathcal{D}_xv_h,\mathcal{D}_yv_h\right)^T,\quad\quad
	\Delta_hv_h :=\nabla_h\cdot\bra{\nabla_hv_h}= \bra{\mathcal{D}_x^2+\mathcal{D}_y^2}v_h.
	\]
	
	In the consistency analysis for the underlying mass conservative problem,
	we introduce a mean-zero space
	$$\mathbb{\mathring V}_{h}:=\big\{v\in\mathbb{V}_{h}\,|\, \myinner{v,1}=0\big\}\subset\mathbb{V}_{h}.$$
	For any grid function $v\in \mathbb{\mathring V}_{h}$ and $p>0$, the discrete fractional Laplacian is introduced\cite{Cheng2016J.Sci.Comput.}
	\[
	\bra{-\Delta_h}^{p}v_h
	:=\sum_{m,n}
	\lambda_{m,n}^{p}\widetilde{v}_{m,n}
	e^{\mathrm{i}\nu}(\mathbf{x}_{h}),
	\] 
	and its inverse is given by
	\[
	\bra{-\Delta_h}^{-p}v_h
	:=\sum_{(m,n)\neq0}
	\lambda_{m,n}^{-p}\widetilde{v}_{m,n}
	e^{\mathrm{i}\nu}(\mathbf{x}_{h}),
	\]
furthermore, the $H^{-p}$ inner product and the discrete $H^{-p}$ norm $\mynormb{\cdot}_{-p}$ are given as 
	\begin{align*}
		\myinner{v,w}_{-p}
		:=\myinnerb{\bra{-\Delta_h}^{-p}v,w},\qquad \mynormb{v}_{-p}:=\sqrt{\myinner{v,v}_{-p}}\,,
	\end{align*}
besides, an application of  discrete Parseval's identity shows
	\begin{align*}
		\myinnerb{(-\Delta_h)^p v,v}=\sum_{(m,n)\neq0}\lambda_{m,n}^p|\widetilde{v}_{m,n}|^2.
	\end{align*}
	
	\subsection{The fully discrete variable-step BDF2 scheme}
	Consider the nonuniform time levels
	$0=t_{0}<t_{1}<\cdots<t_N=T$ with the time-step sizes
	$\tau_{k}:=t_{k}-t_{k-1}$ for $1\le k \le N$,
	and denote the maximum time-step size $\tau:=\max_{1\le k\le N}\tau_k$.
	Let the local time-step ratio $r_k:=\tau_k/\tau_{k-1}$ for $2\le k\le N$,
	and $r_1\equiv0$ when it appears.
	Given a grid functions series $\{v^k\}_{k=0}^N$,
	denote $\diff v^{k}:=v^{k}-v^{k-1}$ and  $\partial_{\tau}v^{k}:=\diff v^{k}/\tau_k$ for $k\geq{1}$.
	Taking $v^n=v(t_n)$, the variable-step BDF2 formula could be viewed as a discrete convolution summation \cite{Liao2021math.NA}
	\begin{align}\label{def: BDF2-Formula}
		D_2v^n:=\sum_{k=1}^nb_{n-k}^{(n)}\diff v^k
		\quad \text{for $n\ge2$},
	\end{align}
	where the discrete convolution kernels $b_{n-k}^{(n)}$ are
	defined by
	\begin{align}\label{def:BDF2-kernels}
		b_{0}^{(n)}:=\frac{1+2r_n}{\tau_n(1+r_n)},\quad
		b_{1}^{(n)}:=-\frac{r_n^2}{\tau_n(1+r_n)}\quad \text{and} \quad
		b_{j}^{(n)}:=0,\quad \mathrm{for}\quad 2\le j\le n-1.
	\end{align}
	Then, the fully implicit BDF2 scheme with variable time-step sizes
	of the space FCH equation \eqref{cont: Problem-FCH} is given as
	\begin{align}\label{scheme:FCH BDF2}
		D_2\phi^n=-\kappa(-\Delta_h)^{\alpha}\mu^n\quad\text{with}\quad\mu^{n}=\epsilon^{2}(-\Delta_h)\phi^{n}+F'(\phi^n),\quad \mathrm{for}\quad 2\le n\le N
	\end{align}
	with the Fourier pseudo-spectral approximation for the spatial discretization.
	
	The first-level numerical solution is computed by the TR-BDF2 method\cite{Liao2021math.NA} which achieves reasonable accuracy and controls possible oscillations around the initial time. It is easy to verify that there exists a positive constant $c_0$,
	depends on the parameter $\Omega$, $\epsilon$ and $\phi^0$, such that
	$$\frac{\epsilon^2}{2}\mynormb{\nabla_h\phi^1}^2+
	\myinnerb{F(\phi^1),1}+ \frac{\tau_{2}}{2\kappa}\mynormb{\partial_{\tau} \phi^1}_{-\alpha}^2 \le c_0.$$
	Besides, we define the discrete version of the free energy functional \eqref{FCH-1 energy} as
	\begin{align}\label{def: discrete free energy}
		E[\phi^n]
		:=\frac{\epsilon^2}{2}\mynormb{\nabla_h\phi^n}^2+\myinnerb{F(\phi^n),1} \quad\text{for $n\ge 1$.}
	\end{align}

	\subsection{Auxiliary lemmas}	
	We introduce the following fundamental lemmas, which are used in the complete theoretical analysis of the proposed algorithm.
	\begin{lemma}\cite[Lemma 3.1]{Cheng2019J.Comput.Appl.Math.}
		For any $v \in \mathbb{V}_{h}$, there exists a constant $c_{\Omega}$ dependent on the problem domain $\Omega$ such that
		\begin{align}\label{ieq: H1 embedding L6}
			\mynormb{v}_{l^{6}} \le c_{\Omega}\left(\mynormb{v}+\mynormb{\nabla_{h}v}\right) .
		\end{align}
		Moreover, for any $v\in\mathbb{\mathring V}_{h}$, we have
		\begin{align}\label{l6}
			\mynormb{v}_{l^{6}} \le c_{z}\mynormb{\nabla_{h}v}.
		\end{align}
	\end{lemma}

 Following the proofs of the integration by parts formula and the Sobolev embedding inequality \cite[Lemma A.1 and Lemma A.3]{Mark2017SIAM}, the following summation by parts formula and the embedding inequality in the discrete sense of the fractional Laplacian are obviously shown.
	\begin{lemma}
		For any grid function $v,w\in \mathbb{V}_{h}$ and $p,q\ge0$, it holds that
		$$\myinnerb{(-\Delta_h)^{p+q}v,w}=\myinnerb{(-\Delta_h)^p v,(-\Delta_h)^q w}.$$
	\end{lemma}
	\begin{lemma}
		For any grid function $v\in \mathbb{ V}_{h}$ and $0\le p\le q$, it holds that
		$$\mynorm{(-\Delta_h)^pv}\le\mynorm{(-\Delta_h)^qv}.$$
	\end{lemma}
	
Let $C_{per}^{\infty}(\Omega)$ be a set of infinitely differentiable $L$-periodic functions defined on $\Omega$. For each $p\ge0$, let $H_{per}^{p}(\Omega)$ be the closure of $C_{per}^{\infty}(\Omega)$ in $H^{p}(\Omega)$,
endowed with the semi-norm $|\cdot|_{H_{per}^p}$ and the norm $\mynorm{\cdot}_{H_{per}^{p}}$. 
For simplicity, we denote $|\cdot|_{H^p}:=|\cdot|_{H_{per}^p}$, $\mynorm{\cdot}_{H^{p}}:=\mynorm{\cdot}_{H_{per}^{p}}$, and $\mynorm{\cdot}_{L^{2}}:=\mynorm{\cdot}_{H^{0}}$. The following properties of the $L^{2}$ projection $P_{M}$ and the  interpolation operator $I_{M}$ are available.
\begin{lemma}\cite{Caputo1982MC}\label{lem:Projection-Estimate}
	For any $v\in{H_{per}^{q}}(\Omega)$ and $0\le{p}\le{q}$, it holds that
	\begin{align}
		\mynorm{P_{M}v-v}_{H^{p}}
		\le C_vh^{q-p}|v|_{H^{q}},
		\quad \mynorm{P_{M}v}_{H^{p}}\le C_v\mynorm{v}_{H^{p}};\label{Projection-Estimate}
	\end{align}
	In addition if $q>1$,
	\begin{align}
		\mynorm{I_{M}v-v}_{H^{p}}
		\le C_vh^{q-p}|v|_{H^{q}},
		\quad \mynorm{I_{M}v}_{H^{p}}\le C_v\mynorm{v}_{H^{p}}.\label{Interpolation-Estimate}
	\end{align}
\end{lemma}

	\section{Solvability and energy dissipation law}
	\setcounter{equation}{0}
	{Now we are ready to introduce the following lemma for proving the unique solvability and the energy stability. }
    {\begin{lemma}\label{lem: Embedding-Inequality}
    	For any grid function $v\in \mathbb{\mathring V}_{h}$, it holds that if $p\le q$,
    	\begin{align}\label{Finequality}
    		\mynormb{v}^2\le \nu^{p-q}\mynormb{(-\Delta_h)^{q/2} v}\mynormb{v}_{-p}\quad\text{with $\nu=2\pi/L$}.
    	\end{align}
    \end{lemma}
    \begin{proof}
    	With the help of Cauchy-Schwarz inequality, it is obviously that
    	\begin{align*}
    		\mynormb{v}^2=\sum_{(m,n)\neq0}|\widetilde{v}_{m,n}|^2
    		&\le\brabb{\sum_{(m,n)\neq0}\lambda_{m,n}^{p}|\widetilde{v}_{m,n}|^2}^{\frac{1}{2}}
    		\brabb{\sum_{(m,n)\neq0}\lambda_{m,n}^{-p}|\widetilde{v}_{m,n}|^2}^{\frac{1}{2}}
    		\\
    		&\le\nu^{p-q}\brabb{\sum_{(m,n)\neq0}\lambda_{m,n}^{q}|\widetilde{v}_{m,n}|^2}^{\frac{1}{2}}
    		\brabb{\sum_{(m,n)\neq0}\lambda_{m,n}^{-p}|\widetilde{v}_{m,n}|^2}^{\frac{1}{2}}
    		\\
    		&\le\nu^{p-q}\mynormb{(-\Delta_h)^{q/2} v}\mynormb{v}_{-p}.
    	\end{align*}
    \end{proof}	}	
\begin{remark}
 Lemma \ref{lem: Embedding-Inequality} extends the discrete H\"older inequality to a general case related to the fractional Laplacian. In other words, when {$p=q=1$}, the above estimate \eqref{Finequality} turns into the classical H\"older inequality $$\mynormb{v}^2\le \mynormb{\nabla_hv}\mynormb{v}_{-1}.$$ 
 Moreover, the inequality will also be in favor of dealing with the general problems deduced by $H^{-p}$ ($p>0$) gradient flow which usually involve the $H^{-p}$ norm in the numerical analysis.
\end{remark}
	\subsection{Mass conservation}
	We review the DOC  kernels $\{\theta_{n-k}^{(n)}\}_{k=2}^n$  \cite{Liao2021math.NA} in the following
	\begin{align}\label{def: DOC-Kernels}
		\theta_{0}^{(n)}:=\frac{1}{b_{0}^{(n)}}
		\;\; \text{for $n\ge2$}\quad \mathrm{and} \quad
		\theta_{n-k}^{(n)}:=-\frac{1}{b_{0}^{(k)}}
		\sum_{j=k+1}^n\theta_{n-j}^{(n)}b_{j-k}^{(j)}
		\;\; \text{for $n\ge k+1\ge3$}.
	\end{align}
	The discrete orthogonal identity gives
	\begin{align}\label{eq: orthogonal identity}
		\sum_{j=k}^n\theta_{n-j}^{(n)}b_{j-k}^{(j)}\equiv \delta_{nk}\quad\text{for $2\le k\le n$,}
	\end{align}
	where $\delta_{nk}$ is the Kronecker delta symbol. 

	\begin{theorem}
		The solution $\phi^n$ of the implicit variable-step BDF2 scheme \eqref{scheme:FCH BDF2} satisfies the mass conservation 
		\begin{align}
			\myinnerb{\phi^n,1}=\myinnerb{\phi^0,1}\quad\text{for}\quad n\ge1.
		\end{align}
	\end{theorem}
	\begin{proof}
			It is easy to check that $\myinnerb{\phi^1,1}=\myinnerb{\phi^0,1}$. In a similar manner, for the fully discrete scheme \eqref{scheme:FCH BDF2}, it holds that
		$$\myinnerb{D_2\phi^j,1}=\myinnerb{-\kappa(-\Delta_h)^\alpha\mu^j,1}=0\quad\text{for $j\ge2$.}$$
		Multiplying both sides of this equality by the DOC kernels
		$\theta_{n-j}^{(n)}$ and summing the index $j$ from $j=2$ to $n$, we get
		$$\sum_{j=2}^n\theta_{n-j}^{(n)}\myinnerb{D_2\phi^j,1}=0\quad\text{for $n\ge2$}.$$
		Observing that, by exchanging the summation order
		and using the identity \eqref{eq: orthogonal identity}, the following equality holds
		\begin{align}\label{eq: orthogonal equality for BDF2}
			\sum_{j=2}^n\theta_{n-j}^{(n)}D_2\phi^j=&\,\sum_{j=2}^{n}\theta_{n-j}^{(n)}b_{j-1}^{(j)}\diff \phi^1
			+\sum_{j=2}^{n}\theta_{n-j}^{(n)}
			\sum_{k=2}^{j}b_{j-k}^{(j)}\diff \phi^k\nonumber\\
			=&\,\theta_{n-2}^{(n)}b_{1}^{(2)}\diff \phi^1+\diff \phi^n\quad\text{for $n\ge2$.}
		\end{align}
		Owing to the fact that $\myinnerb{\phi^1,1}=\myinnerb{\phi^0,1}$, we obtain $\myinnerb{\diff \phi^n,1}=0$ directly. 
	\end{proof}
	\subsection{Unique solvability}
	\begin{theorem}\label{thm: convexity solvability}
		If the time step size satisfies $\tau_n\le\frac{4\epsilon^2(1+2r_n)}{\kappa\nu^{2\alpha-2}(1+r_n)}$,
		the implicit variable-step BDF2 scheme \eqref{scheme:FCH BDF2} for the space FCH is uniquely solvable.
	\end{theorem}
	\begin{proof}
		For any fixed time level indexes $n\ge2$,
		we consider the following energy functional $G$ on the space
		$\mathbb{V}_{h}^{*}:=\big\{z\in\mathbb{V}_{h}\,|\, \myinnerb{z,1}=\myinnerb{\phi^{n-1},1}\big\},$
		\begin{align*}
			G[z]:=\frac{b_0^{(n)}}{2}\mynormb{z-\phi^{n-1}}_{-\alpha}^2+b_1^{(n)}\myinnerb{\diff \phi^{n-1},z-\phi^{n-1}}_{-\alpha}
			+\frac{\kappa\epsilon^2}{2}\|\nabla_h z\|^2+\kappa\myinnerb{F(z),1}.
		\end{align*}
		According to the definition of $F(z)$ and the Young's inequality, we get
		\begin{align*}
			G[z] &\ge \frac{b_0^{(n)}}{2}\mynormb{z-\phi^{n-1}}_{-\alpha}^2+b_1^{(n)}\myinnerb{\diff \phi^{n-1},z-\phi^{n-1}}_{-\alpha}+\frac{\kappa}{4}\mynormb{z^2-1}^2\\
			&\ge -\frac{(b_1^{(n)})^2}{2b_0^{(n)}}\mynormb{\diff \phi^{n-1}}_{-\alpha}^2+\frac{\kappa}{4}\mynormb{z^2-1}^2.
		\end{align*}
		Then, the energy functional $G$ is coercive on the space
		$\mathbb{V}_{h}^{*}$ which implies $G$ has a minimizer. Meanwhile, the functional $G$ is strictly convex since for any $\lambda\in \mathbb{R}$ and any $\psi\in \mathbb{\mathring V}_{h}$, it leads to
		\begin{align*}
			\frac{\zd^2G}{\zd\lambda^2}[z+\lambda\psi]\Big|_{\lambda=0}=&\,b_0^{(n)}\mynormb{\psi}_{-\alpha}^2
			+\kappa\epsilon^2\mynormb{\nabla_h\psi}^2
			+3\kappa\mynormb{z\psi}^2-\kappa\mynormb{\psi}^2\\
			\ge& \frac{\kappa\nu^{2\alpha-2}}{4\epsilon^2}\mynormb{\psi}_{-\alpha}^2
			+\kappa\epsilon^2\mynormb{\nabla_h\psi}^2
			+3\kappa\mynormb{z\psi}^2-\kappa\mynormb{\psi}^2\\
			\ge&\,\kappa\nu^{\alpha-1}\mynormb{\nabla_h\psi}\mynormb{\psi}_{-\alpha}-\kappa\mynormb{\psi}^2+3\kappa\mynormb{z\psi}^2>0,
		\end{align*}
		in which the second step comes from the time-step size condition, while the last step is based on Lemma \ref{lem: Embedding-Inequality}. Thus the functional $G$ has a unique minimizer, denoted by $\phi^n$, if and only if it solves the equation
		\begin{align*}
			0=\frac{\zd G}{\zd\lambda}[z+\lambda\psi]\Big|_{\lambda=0}=&\,
			\myinnerb{b_0^{(n)}\brat{z-\phi^{n-1}}+b_1^{(n)}\diff \phi^{n-1},\psi}_{-\alpha}
			+\kappa\myinnerb{-\epsilon^2\Delta_h z+F'(z),\psi}\\
			=&\,
			\myinnerB{b_0^{(n)}\brat{z-\phi^{n-1}}+b_1^{(n)}\diff \phi^{n-1}+
				\kappa(-\Delta_h)^\alpha(F'(z)-\epsilon^2\Delta_h z),\psi}_{-\alpha}.
		\end{align*}
		The above equation holds for any $\psi\in \mathbb{\mathring V}_{h},$ indicating that the unique minimizer $\phi^n\in\mathbb{V}_{h}^{*}$ yields
		\begin{align*}
			b_0^{(n)}\diff \phi^n+b_1^{(n)}\diff \phi^{n-1}+
			\kappa(-\Delta_h)^\alpha(F'(\phi^n)-\epsilon^2\Delta_h\phi^n)=0,
		\end{align*}
		which corresponds to the implicit BDF2 scheme \eqref{scheme:FCH BDF2}.
		It verifies the claimed result.
	\end{proof}
	
	\subsection{Energy dissipation law}
	The modified energy dissipation law is presented by using the following discrete gradient structure of variable-step BDF2 method.
	
	\begin{lemma}\cite[Lemma 2.1]{Liao2021math.NA}\label{lem:conv kernels positive}
		Let $0<r_k\le r_{\mathrm{user}}(<4.864)$ for $2\le k\le N$. For any real sequence $\{w_k\}_{k=1}^n$ with n entries,
		it holds that
		\begin{align*}
			2w_k\sum_{j=1}^kb_{k-j}^{(k)}w_j
			&\ge\frac{r_{k+1}^{3/2}}{1+r_{k+1}}\frac{w_k^2}{\tau_k}
			-\frac{r_k^{3/2}}{1+r_k}\frac{w_{k-1}^2}{\tau_{k-1}}
			+R(r_k,r_{k+1})
			\frac{w_k^2}{\tau_k}\quad\text{for $k\ge2.$}
		\end{align*}
		So the discrete convolution kernels $b_{n-k}^{(n)}$ are positive definite,
		\[
		2\sum_{k=2}^n w_k \sum_{j=2}^k b_{k-j}^{(k)}w_j\ge
		\sum_{k=2}^n R(r_k,r_{k+1})
		\frac{w_k^2}{\tau_k}\quad\text{for $n\ge 2$},
		\]
		where
		\[R(a,b):=\frac{2+4a-a^{3/2}}{1+a}-\frac{b^{3/2}}{1+b}\ge\min\{R(0,r_{\mathrm{user}}),R(r_{\mathrm{user}},r_{\mathrm{user}})\}>0\quad
		\text{for $0< a,b\le r_{\mathrm{user}}.$}\]
	\end{lemma}
	
	We introduce a modified discrete energy
	\begin{align}\label{def: modified discrete energy}
		\mathcal{E}[\phi^n]
		:=E[\phi^n] + \frac{\sqrt{r_{n+1}}\tau_{n+1}}{2\kappa(1+r_{n+1})}\mynormb{\partial_{\tau} \phi^n}_{-\alpha}^2
		\quad\text{for $n\ge 1$}.
	\end{align}	
	
	\begin{theorem}\label{thm:energy decay law}
		Assume that $0<r_k\le r_{\mathrm{user}}(<4.864)$ and the time-step sizes satisfy
		\begin{align}\label{Restriction-Time-Step}
			\tau_n \le \frac{4\epsilon^2}{\kappa\nu^{2\alpha-2}}\min\Big\{\frac{1+2r_n}{1+r_n},R(r_n,r_{n+1})
			\Big\}
			\quad\text{for $n\ge 1$,}
		\end{align}
		then the implicit variable-step BDF2 scheme \eqref{scheme:FCH BDF2} preserves the following energy dissipation law
		\begin{align*}
			\mathcal{E}[\phi^n] \le \mathcal{E}[\phi^{n-1}]
			\quad\text{for $n\ge 2$.}
		\end{align*}
	\end{theorem}
	\begin{proof}
		Taking the inner product of \eqref{scheme:FCH BDF2} by $(-\Delta_h)^{-\alpha}\diff \phi^{n}/\kappa$ gives
		\begin{align}\label{Energy-Law-Inner}
			\frac{1}{\kappa}\myinnerb{D_2\phi^n,\diff \phi^{n}}_{-\alpha}
			-\epsilon^2\myinnerb{\Delta_h\phi^n,\diff \phi^{n}}
			+\myinnerb{F'(\phi^n),\diff \phi^{n}}=0.
		\end{align}
		For the first term, we make use of Lemma \ref{lem:conv kernels positive} and obtain
		\begin{align}\label{energyfirstterm}
			\frac{1}{\kappa}\myinnerb{D_2\phi^n,\diff \phi^{n}}_{-\alpha}
			\ge&\, \frac{\sqrt{r_{n+1}}\tau_{n+1}}{2\kappa(1+r_{n+1})}\mynormb{\partial_\tau\phi^n}_{-\alpha}^2
			-\frac{\sqrt{r_{n}}\tau_{n}}{2\kappa(1+r_n)}\mynormb{\partial_\tau\phi^{n-1}}_{-\alpha}^2\nonumber\\
			&+\frac{1}{2\kappa\tau_n}R(r_n,r_{n+1})\mynormb{\diff\phi^n}_{-\alpha}^2.
		\end{align}
		On the other hand, the summation by parts formula and the equality $2a(a-b)=a^2-b^2+(a-b)^2$  imply that
		\begin{align}\label{energysecondterm}
			-\epsilon^2\myinnerb{\Delta_h\phi^n,\diff \phi^{n}}&=\epsilon^2\myinnerb{\nabla_h\phi^n,\diff\nabla_h\phi^n}\nonumber\\
			&=\frac{\epsilon^2}{2}\mynormb{\nabla_h\phi^n}^2
			-\frac{\epsilon^2}{2}\mynormb{\nabla_h\phi^{n-1}}^2
			+\frac{\epsilon^2}{2}\mynormb{\diff\nabla_h\phi^n}^2.
		\end{align}
		Meanwhile, with an application of Lemma \ref{lem: Embedding-Inequality}, and based on the fact that
		\begin{align*}
			\brab{a^3-a}\bra{a-b}= \frac{1}{4}\bra{a^2-1}^2 - \frac{1}{4}\bra{b^2-1}^2-\frac{1}{2}\bra{a-b}^2 +\frac{1}{2}a^2\bra{a-b}^2+\frac{1}{4}\bra{a^2-b^2}^2,
		\end{align*}
		the third term in \eqref{Energy-Law-Inner} could be handled as follows 
		\begin{align}\label{nonlinear}
			\myinnerb{F'(\phi^n),\diff \phi^n}
			&\ge \myinnerb{F(\phi^n),1}-\myinnerb{F(\phi^{n-1}),1}
			-\frac{1}{2}\mynormb{\diff \phi^n}^2\nonumber\\
			&\ge \myinnerb{F(\phi^n),1}-\myinnerb{F(\phi^{n-1}),1}-\frac{\nu^{\alpha-1}}{2}\mynormb{\nabla_h\diff\phi^n}\mynormb{\diff \phi^n}_{-\alpha}\nonumber\\
			&\ge \myinnerb{F(\phi^n),1}-\myinnerb{F(\phi^{n-1}),1}-\frac{\epsilon^2}{2}\mynormb{\nabla_h\diff\phi^n}^2-\frac{\nu^{2\alpha-2}}{8\epsilon^2}\mynormb{\diff\phi^n}^2_{-\alpha}.
		\end{align}  
		Subsequently, a substitution of \eqref{energyfirstterm}-\eqref{nonlinear} into \eqref{Energy-Law-Inner} yields
		\begin{align*}
			\brabb{\frac{1}{2\kappa\tau_n}R(r_n,r_{n+1})-\frac{\nu^{2\alpha-2}}{8\epsilon^2}}\mynormb{\diff\phi^n}^2_{-\alpha}+\mathcal{E}[\phi^n] \le \mathcal{E}[\phi^{n-1}].
		\end{align*}
		The result follows with the help of the time-step size condition \eqref{Restriction-Time-Step}.
	\end{proof}
		\begin{remark}
		To perform the energy stability of the resulting scheme as excepted, we need to restrict the time-step sizes $\tau_n=\mathcal{O}(\frac{\epsilon^2}{\kappa})$. The first condition of the time-step size \eqref{Restriction-Time-Step} stems from  the requirement of the unique solvability.  This is because the unique solvability problem of the fully implicit BDF2 scheme is equivalent to solving
		\begin{align*}
			\phi^{n}=\arg \min_{z\in\mathbb{V}_{h}^{*}} \frac{b_0^{(n)}}{2}\mynormb{z-\phi^{n-1}}_{-\alpha}^2+b_1^{(n)}\myinnerb{\diff \phi^{n-1},z-\phi^{n-1}}_{-\alpha}+\kappa E[z].
		\end{align*}
		We need to adjust the time-step size in $b_0^{(n)}$ to obtain strong convexity such that it can balance the concave, negative quadratic term in the nonlinear term. 
		
		Nevertheless, we would like to clarify that such a constraint \eqref{Restriction-Time-Step} is reasonable when we choose suitable $\epsilon$ and $\kappa$ in the numerical experiments. For simplicity, we fix $\Omega=(0,2\pi)^{2}$ and $r_{\mathrm{user}}=4$.
		\begin{itemize}
			\item If the time-step ratios $0<r_n, r_{n+1} \leq 2$ is available, then the time-step size satisfies $\tau_n\le\frac{4\epsilon^2}{\kappa}\min\{1,R_L(0,2)\}=\frac{4\epsilon^2}{\kappa}$;
			\item If the current time-step ratio $2<r_n \leq 3$, we take the next one to be $0<r_{n+1} \leq 2$, then the time-step size satisfies $\tau_n\le\frac{4\epsilon^2}{\kappa}\min\{\frac{5}{3},R_L(3,2)\}\le\frac{5\epsilon^2}{\kappa}$;
			\item If the ratio $3<r_n \leq r_{\text {user}}$, we will set a small $\tau_{n+1}$ or $r_{k+1}$ to ensure numerical accuracy. For instance, if we adopt the next time-step ratio to be $0<r_n \leq 1$, the time-step size satisfies $\tau_{n}\le\frac{4\epsilon^2}{\kappa}\min\{\frac{7}{4},R_L(r_{\mathrm{user}},1)\}=\frac{6\epsilon^2}{\kappa}$.
		\end{itemize}
	To be more specific, the time-step sizes are required to be bounded by 0.16 if we take $\epsilon=0.02$ and $\kappa=0.01$. In a nutshell, the constraint of time-step size \eqref{Restriction-Time-Step} is reasonable with proper $\epsilon$ and $\kappa$.
	\end{remark}
   
	\begin{remark}
		With an application of Lemma \ref{lem: Embedding-Inequality}, we could also establish the energy dissipation law for the generalized fractional variant of the Cahn-Hilliard equation \eqref{general FCH} by viewing \eqref{general FCH} as the $H^{-\alpha}$ gradient flow associated with the non-local free energy functional $E_2[\Phi]$. Taking the corresponding modified discrete energy 
		\begin{align*}
			\mathcal{E}'[\phi^k]
			:=\frac{\epsilon^2}{2}\mynormb{(-\Delta_h)^{\frac{\beta}{2}}\phi^k}^2+\myinnerb{F(\phi^k),1}+ \frac{\sqrt{r_{k+1}}\tau_{k+1}}{2\kappa(1+r_{k+1})}\mynormb{\partial_{\tau} \phi^k}_{-\alpha}^2,
		\end{align*}
and following the similar proof in Theorem \ref{thm:energy decay law}, the implicit variable-step BDF2 scheme for the generalized fractional variant of the Cahn-Hilliard equation \eqref{general FCH} preserves the modified energy dissipation law provided the fractional orders satisfying $\alpha\le\beta$. 
	\end{remark}
	
A key issue of the error analysis for the nonlinear differential equations is to control the nonlinear term of numerical solution, and for which we deduce the boundedness of the numerical solution from the discrete energy dissipation law. The proof is quite straightforward, we omit it here for brevity.

	\begin{lemma}\label{lem: bound-BDF2 Solution FCH}
		If $0<r_k\le r_{\mathrm{user}}(<4.864)$ and the step sizes $\tau_n$ satisfy \eqref{Restriction-Time-Step},
		the following bound of the numerical solution is available
		\begin{align*}
			\mynormb{\phi^n}+\mynormb{\nabla_h\phi^n}\le c_1:=\sqrt{4\epsilon^{-2}c_0+(2+\epsilon^2)\abs{\Omega_h}}\quad\text{for $n\ge2$.}
		\end{align*}
	\end{lemma}
	
	\section{Convergence analysis}
	\setcounter{equation}{0}
	In this section, we derive the convergence analysis for the fully discrete scheme \eqref{scheme:FCH BDF2}. The following discrete inequality provides an estimate for the discrete $L^6$ norm.
	\begin{lemma}
		For any grid function $v\in \mathbb{\mathring V}_{h}$ and $p>0$, the following inequality holds	
		\begin{align}\label{ieq: H1 embedding H1a}
			\mynormb{v}_{l^{6}} \le c_{z}\mynormB{\left(-\Delta_{h}\right)^{\frac{1+p}{2}}v}^\frac{1}{1+p}\mynormb{v}^\frac{p}{1+p}.
		\end{align}
	\end{lemma}
	\begin{proof}
		According to the discrete H\"{o}lder inequality, it follows that
		\begin{align*}\label{1}
			\mynormb{\nabla_{h}v}^2&=\sum_{(m,n)\neq0}\lambda_{m,n}|\widetilde{v}_{m,n}|^2=\sum_{(m,n)\neq0}\lambda_{m,n}|\widetilde{v}_{m,n}|^{\frac{2}{1+p}}|\widetilde{v}_{m,n}|^{\frac{2p}{1+p}}\\
			&\le \kbrabb{\sum_{(m,n)\neq0}\brab{\lambda_{m,n}|\widetilde{v}_{m,n}|^{\frac{2}{1+p}}}^{1+p}}
			^{\frac{1}{1+p}}
			\kbrabb{\sum_{(m,n)\neq0}\brab{|\widetilde{v}_{m,n}|^{\frac{2p}{1+p}}}^{\frac{1+p}{p}}}^{\frac{p}{1+p}}\\
			&=\kbrabb{\sum_{(m,n)\neq0}\lambda_{m,n}^{1+p}|\widetilde{v}_{m,n}|^2}^{\frac{1}{1+p}}
			\kbrabb{\sum_{(m,n)\neq0}|\widetilde{v}_{m,n}|^2}^{\frac{p}{1+p}}\\
			&=\mynormB{\left(-\Delta_{h}\right)^{\frac{1+p}{2}}v}^\frac{2}{1+p}\mynormb{v}^\frac{2p}{1+p}.
		\end{align*}
With the help of the discrete embedding inequality \eqref{l6}, it completes the proof.
	\end{proof}
	
Here and hereafter, we denote
$\sum_{k,j}^{n,k}:=\sum_{k=2}^n\sum_{j=2}^k$ for the simplicity of presentation. The properties of the DOC kernels are introduced as follows.
	\begin{lemma}\cite[Lemma 3.1]{Liao2021math.NA}\label{lem: DOC property}
		If $0<r_k\le r_{\mathrm{user}}(<4.864)$ holds, the DOC kernels $\theta_{n-j}^{(n)}$ defined in \eqref{def: DOC-Kernels} satisfy:
		\begin{itemize}
			\item[(i)] The discrete kernels $\theta_{n-j}^{(n)}$ are positive definite;
			\item[(ii)] 
			$\displaystyle \theta_{n-j}^{(n)}=\frac{1}{b^{(j)}_{0}}\prod_{i=j+1}^n\frac{r_i^2}{1+2r_i}$
			for $2\le j\le n$;
		\item[(iii)] 
		$\displaystyle\sum_{k,j}^{n,k}\theta_{k-j}^{(k)}\le t_n$ for $n\ge2$.
	\end{itemize}
\end{lemma}

We introduce the discrete Young-type convolution inequality for proving the convergence.  

\begin{lemma}\label{lem:DOC quadr form Ha-H1+a embedding inequ}
	Assume that $u^k\in \mathbb{V}_{h}$, $v^k\in \mathbb{\mathring V}_{h}$ $(2\le k\le n)$
	and there exists a constant $c_u$ such that
	$\mynormb{u^k}_{l^3}\le c_u$ for $2\le k\le n$. If $0<r_k\le r_{\mathrm{user}}(<4.864)$  holds, then for any $\varepsilon >0$ and $\alpha\in(0,1)$, we obtain
	\begin{align*}
		\sum_{k,j}^{n,k}\theta_{k-j}^{(k)} \myinnerb{u^jv^j,-(-\Delta_h)^{\alpha} v^k}
		\le\varepsilon\sum_{k,j}^{n,k}\theta_{k-j}^{(k)}\myinnerb{(-\Delta_h)^{\frac{1+\alpha}{2}} v^j,(-\Delta_h)^{\frac{1+\alpha}{2}} v^k}
		+\eta
		\sum_{k=2}^n\tau_k  \mynormb{v^k}^2,
	\end{align*}
	where $\eta:=\frac{c_z^2c_u^2\mathfrak{m}_2\mathfrak{m}_3\alpha}{\varepsilon\mathfrak{m}_1^2(1+\alpha)\varepsilon_3^{\frac{1+\alpha}{\alpha}}}$, $\varepsilon_3^{1+\alpha}:=\frac{\varepsilon^2\mathfrak{m}_1^3(1+\alpha)}{4c_z^2c_u^2\mathfrak{m}_2^2\mathfrak{m}_3}$.
\end{lemma}
The proof is shown in the Appendix. 
\begin{remark}
Lemma \ref{lem:DOC quadr form Ha-H1+a embedding inequ} gives an efficient inequality in dealing with the nonlinear term and the fractional Laplacian, it could be extended to the general problems with fractional Laplacian. Furthermore, if the original problem \eqref{cont: Problem-FCH} is computed using the higher order backward difference formula, the similar result is also valid when the corresponding DOC kernels are positive definite.  
\end{remark}

We are prepared to prove the convergence for the implicit variable-step BDF2 scheme \eqref{scheme:FCH BDF2}.
\begin{theorem}\label{thm: L2 Convergence-CH}
	Assume that the space FCH problem \eqref{cont: Problem-FCH} has a solution $\Phi\in C^3\brab{[0,T];{H}_{per}^{s+2+2\alpha}}$
	for $s\ge0$. Denote $\Phi^n:=\Phi(t_n)$ and suppose that the maximum step size $\tau\le 1/c_{\eta}$ holds, then the solution $\phi^n$ is convergent in the $L^2$ norm,
	\begin{align*}
		\mynormb{\Phi^n-\phi^n}
		\le C_{\Phi}\exp\brab{c_{\eta}t_{n-1}}&\,\Big(
		\mynormb{\Phi_M^1-\phi^1}+\tau\mynorm{\partial_{\tau}\brab{\Phi_M^1-\phi^1}}\\
		&\,+t_nh^s+t_n\tau^2\max_{0\le t\le T}\mynormb{\Phi'''(t)}_{L^2}\Big)\quad\text{for $2\le n\le N$,}
	\end{align*}
	where $c_{\eta}:=\frac{4\kappa c_z^2c_3^2\mathfrak{m}_2\mathfrak{m}_3\alpha}{\varepsilon^2\mathfrak{m}_1^2(1+\alpha)\varepsilon_3^{\frac{1+\alpha}{\alpha}}}$, $\varepsilon_3^{1+\alpha}:=\frac{\varepsilon^4\mathfrak{m}_1^3(1+\alpha)}{4c_z^2c_3^2\mathfrak{m}_2^2\mathfrak{m}_3}$
		and $c_3:=c_{\Omega}^2(c_1^2+c_1c_2+c_2^2)+\absb{\Omega_h}^{1/3}$.
\end{theorem}

\begin{proof}
	Let $\Phi_M^n:=\brab{P_M\Phi}(\cdot,t_n)$
	be the $L^2$ projection of the exact solution at time $t=t_n$. Denoting $e^n:=\Phi_M^n-\phi^n$, we have 
	\begin{align}\label{Triangle-Projection-Estimate}
		\mynorm{\Phi^n-\phi^n}=\mynorm{\Phi^n-\Phi_M^n+e^n}\le\mynorm{\Phi^n-\Phi_M^n}+\mynorm{e^n}\quad \text{for $1\le n\le N$}.
	\end{align}
	By virtue of Lemma \ref{lem:Projection-Estimate}, it shows that
	$$\mynorm{\Phi^n-\Phi_M^n}=\mynorm{I_M\Phi^n-\Phi_M^n}_{L^2}
	\le \mynorm{I_M\Phi^n-\Phi^n}_{L^2}+\mynorm{\Phi^n-\Phi_M^n}_{L^2}
	\le C_{\phi}h^{s}\absb{\Phi^n}_{H^{s}}.$$
	Noticing that the projection solution $\Phi_M^n\in S_{M}$, it gives
	\[
	\myinnerb{\Phi_M^n,1}
	=\myinnerb{\Phi_M^0,1}=\myinnerb{\phi^0,1}
	=\myinnerb{\phi^n,1},
	\]
	which implies the error function $e^n\in \mathbb{\mathring V}_{h}$. Then, the discussion of the upper bound of $\mynorm{e^n}$ is carried out in three steps.
	
	\noindent\textbf{Step 1: Consistency analysis of the spatial discretization}
	
	The projection solution $\Phi_M$ satisfies the following spatially semi-discrete equation
	\begin{align}\label{Projection-Equation}
		\partial_t\Phi_M
		=-\kappa(-\Delta_h)^\alpha\mu_M
		+\zeta_{P}\quad\text{with}\quad
		\mu_M=F'(\Phi_M)-\epsilon^2\Delta_h\Phi_M,
	\end{align}
	where the spatial consistency error $\zeta_{P}(\myvec{x}_h,t)$ is given as
	\begin{align}\label{Projection-truncation error}
		\zeta_{P}:=\partial_t\Phi_M-\partial_t\Phi
		-\kappa\left((-\Delta)^\alpha\mu - (-\Delta_h)^\alpha\mu_M\right)\quad \text{for $\myvec{x}_h\in\Omega_{h}$.}
	\end{align}
	An application of the triangle inequality gives
	\begin{align}\label{space error}
		\mynorm{\zeta_{P}}
		\le \mynorm{\partial_t\Phi_M-\partial_t\Phi}
		+\kappa\mynorm{(-\Delta)^\alpha\mu - (-\Delta_h)^\alpha\mu_M}.
	\end{align}

	Using Lemma \ref{lem:Projection-Estimate}, the first term on the right hand side of \eqref{space error} leads to	\begin{align}\label{Time-Projection}
		\mynorm{\partial_t\Phi_M-\partial_t\Phi}
		\le\, \mynorm{\partial_t\bra{\Phi_M-\Phi}}_{L^2}
		+\mynorm{\partial_t\Phi-I_M\partial_t\Phi}_{L^2}\le\,  C_\Phi h^s\mynorm{\partial_t\Phi}_{H^s}.
	\end{align}
    For the nonlinear term of the second term on the right hand side of \eqref{space error}, we obtain 
	\begin{align}\label{Nonlinear-Projection-1}
		\mynorm{(-\Delta)^\alpha\Phi^3-(-\Delta_h)^\alpha\Phi_M^3}
		\le \mynorm{(-\Delta)^\alpha\bra{\Phi^3-\Phi_M^3}}
		+ \mynorm{(-\Delta)^\alpha\Phi_M^3-(-\Delta_h)^\alpha\Phi_M^3}.
	\end{align}
	For the first term on the right hand side of \eqref{Nonlinear-Projection-1}, an application of the triangle inequality and the Sobolev embedding inequality indicates 
	\begin{align}\label{Nonlinear-Projection-2}
		&\mynorm{(-\Delta)^\alpha\bra{\Phi^3-\Phi_M^3}}
		=\mynorm{I_M\bra{(-\Delta)^\alpha\bra{\Phi^3-\Phi_M^3}}}_{L^2}
		\nonumber\\
		&\le\mynorm{I_M\bra{(-\Delta)^\alpha\Phi^3}-(-\Delta)^\alpha\Phi^3}_{L^2}+\mynorm{(-\Delta)^\alpha\bra{\Phi^3-\Phi_M^3}}_{L^2}+\mynorm{(-\Delta)^\alpha\Phi_M^3-I_M\bra{(-\Delta)^\alpha\Phi_M^3}}_{L^2}
		\nonumber\\
		&\le C_\Phi h^m\mynorm{\Phi^3}_{H^{s+2\alpha}}+C_\Phi h^s\mynorm{\Phi_M^3}_{H^{s+2\alpha}}+\mynorm{\bra{\Phi^2+\Phi\Phi_M+\Phi_M^2}\bra{\Phi-\Phi_M}}_{H^{2\alpha}}
		\nonumber\\
		&\le C_\Phi h^m\mynorm{\Phi}_{H^{s+2+2\alpha}}^3+C_\Phi h^s\mynorm{\Phi_M}_{H^{s+2+2\alpha}}^3+C_\Phi h^s\mynorm{\Phi}_{H^{s+2+2\alpha}}^3\le C_\Phi h^s.
	\end{align}
	Similarly, the second term on the right hand side of \eqref{Nonlinear-Projection-1} gives
	$\mynorm{(-\Delta)^\alpha\Phi_M^3-(-\Delta_h)^\alpha\Phi_M^3}\le C_\Phi h^s$.
	Thereafter, it holds that	
	\begin{align}\label{nonlinearzong}
	\mynorm{\Delta\Phi^3-\Delta_h\Phi_M^3}\le C_\Phi h^s.
	\end{align}
While, for $p=\alpha,1+\alpha, \alpha\in(0,1)$, with an application of Lemma \ref{lem:Projection-Estimate}, the linear term of the second term on the right hand side of \eqref{space error} yields
	\begin{align}\label{Space-Projection}
			\mynorm{(-\Delta_h)^p\bra{\Phi_M-\Phi}}
			&=\mynormb{I_M\bra{(-\Delta)^p\bra{\Phi_M-\Phi}}}_{L^2}
			\le C_\Phi h^s\mynorm{\Phi}_{H^{s+2p}}.
	\end{align}

	A combination of \eqref{nonlinearzong}
	and \eqref{Space-Projection} leads to
	$\mynorm{(-\Delta)^\alpha\mu - (-\Delta_h)^\alpha\mu_M}\le C_\Phi h^s$. Subsequently, $\mynorm{\zeta_{P}}\le C_\Phi h^s$
	and $\mynorm{\zeta_{P}(t_j)}\le C_\Phi h^s$ for $j\ge2$.
	Utilizing the Lemma \ref{lem: DOC property}, we conclude that
	\begin{align}\label{Projection-consistency}
		\sum_{k=2}^n\mynormb{\Upsilon_{P}^k}\le C_\Phi h^s \sum_{k,j}^{n,k}\theta_{k-j}^{(k)}
		\le C_\Phi t_nh^s\quad\text{where}\;\; \Upsilon_{P}^k:=\sum_{j=2}^k\theta_{k-j}^{(k)}\zeta_{P}(t_j)\;\;\text{for $k\ge2$.}
	\end{align}
	
	\noindent\textbf{Step 2: Local consistency error of the  temporal discretization}
	
	Applying the BDF2 formula to the projection equation \eqref{Projection-Equation}, we obtain the following approximation equation
	\begin{align}\label{Discrete-Projection-Equation}
		D_2\Phi_M^n
		=-\kappa(-\Delta_h)^\alpha\mu_M^n +\zeta_{P}^n + \xi_{\Phi}^n,
	\end{align}
	where $\xi_{\Phi}^n:=D_2\Phi(t_n)-\partial_t\Phi(t_n)$ is the local consistency error of BDF2 formula at the time $t=t_n$. 
	Consider the following convolutional consistency error $\Xi_{\Phi}^k$
	\begin{align}\label{def: BDF2-global consistency}
		\Xi_{\Phi}^k:=\sum_{j=2}^k\theta_{k-j}^{(k)}\xi_{\Phi}^j=
		\sum_{j=2}^k\theta_{k-j}^{(k)}\bra{D_2\Phi(t_j)-\partial_t\Phi(t_j)}\quad\text{for $k\ge2$,}
	\end{align}
	which satisfies
	\begin{align*}\label{lem: BDF2-Consistency-Error}
			\sum_{k=2}^n\absb{\Xi_{\Phi}^k}
			\le&\, 2t_n\max_{2\le j\le n}\braB{\tau_{j}\int_{t_{j-2}}^{t_j} \absb{\Phi'''(s)} \zd{s}}\quad\text{for $n\ge2$.}
	\end{align*}
	Moreover, Lemma \ref{lem:Projection-Estimate} together with
	$\mynormb{\partial_{t}^3\Phi}=\mynormb{I_M\partial_{t}^3\Phi}_{L^2}\le C_\Phi\mynormb{\partial_{t}^3\Phi}_{L^2}$ gives
	\begin{align}\label{temporal error}
		\sum_{k=2}^n\mynormb{\Xi_{\Phi}^k}\le C_\Phi t_n\tau^2\max_{0\le t\le T}\mynormb{\Phi'''(t)}_{L^2}.
	\end{align}
	\noindent\textbf{Step 3: $L^2$ norm error estimate of the fully implicit BDF2 scheme}
	
	Noting that $e^n=\Phi_M^n-\phi^n$, subtracting \eqref{scheme:FCH BDF2} from \eqref{Discrete-Projection-Equation},
	the error equation is established as
	\begin{align}\label{Error-Equation}
		D_2e^n
		=-\kappa(-\Delta_h)^\alpha\brab{-\epsilon^2\Delta_he^n +f_{\phi}^ne^n}
		+\zeta_{P}^n+\xi_{\Phi}^n\quad\text{for $2\le n\le N$,}
	\end{align}
	where the nonlinear term $f_{\phi}^n :=\brat{\Phi_M^n}^2+\Phi_M^n\phi^n+\brat{\phi^n}^2-1$. As a consequence of the discrete embedding inequality \eqref{ieq: H1 embedding L6}, we obtain
	\begin{align}\label{H1 bound projection}
		\mynormb{\Phi_M^n}_{l^6}\le c_\Omega\left(\mynormb{\Phi_M^n}+\mynormb{\nabla_h\Phi_M^n}\right)
			\le c_\Omega\sqrt{2}\mynormb{P_M\Phi^n}_{H^1}\le c_2\quad\text{for $1\le n\le N$,}
	\end{align}
    where Lemma \ref{lem:Projection-Estimate} and the energy dissipation law \eqref{cont:energy dissipation} are used in the last inequality. In turn, utilizing Lemma \ref{lem: bound-BDF2 Solution FCH} and \eqref{ieq: H1 embedding L6}, we have
	\begin{align}\label{ieq: L3 bound nonlinear}
		\mynormb{f_{\phi}^n}_{l^3}\le&\, \mynormb{\Phi_M^n}_{l^6}^2
		+\mynormb{\Phi_M^n}_{l^6}\mynormb{\phi^n}_{l^6}+\mynormb{\phi^n}_{l^6}^2+\absb{\Omega_h}^{1/3}\nonumber\\
		\le&\, c_{\Omega}^2(c_2^2+c_2c_1+c_1^2)+\absb{\Omega_h}^{1/3}=c_3,
	\end{align}
	Multiplying both sides of equation \eqref{Error-Equation} by the DOC kernels $\theta_{k-n}^{(k)}$,
	and summing up $n$ from $n=2$ to $k$, similarly to the proof of equality \eqref{eq: orthogonal equality for BDF2}, we have 
	\begin{align}\label{Error-Equation-DOC}
		\diff e^k
		=-\theta_{k-2}^{(k)}b_{1}^{(2)}\diff e^1-\kappa\sum_{j=2}^k\theta_{k-j}^{(k)}(-\Delta_h)^\alpha\brab{-\epsilon^2\Delta_he^j + f_{\phi}^j e^j}
		+\Upsilon_{P}^k+\Xi_{\Phi}^k
	\end{align}
	for $2\le k\le N$, where $\Upsilon_{P}^k$ and $\Xi_{\Phi}^k$ are
	defined by \eqref{Projection-consistency} and \eqref{def: BDF2-global consistency}, respectively.
	Taking the inner product of \eqref{Error-Equation-DOC} with $2e^k$ on both sides,
	and summing $k$ from 2 to $n$, the following estimate is available
	\begin{align}\label{Error-Equation-Inner}
		\mynormb{e^n}^2\le\mynormb{e^1}^2-2\sum_{k=2}^n\theta_{k-2}^{(k)}b_{1}^{(2)}\mynormb{e^k}\mynormb{\diff e^1}
		+Q^n +2\sum_{k=2}^n\myinnerb{\Upsilon_{P}^k+\Xi_{\Phi}^k,e^k},
	\end{align}
	where $Q^n$ is of the following form
	\begin{align}\label{Error-quadratic forms}
		Q^n:=&\,2\kappa\sum_{k,j}^{n,k}\theta_{k-j}^{(k)}\myinnerb{-\epsilon^2\Delta_he^j +f_{\phi}^je^j,-(-\Delta_h)^\alpha e^k}.
	\end{align}
	Taking $u^j:=f_{\phi}^j$ (with the upper bound $c_u:=c_3$), $v^j:=e^j$ and $\varepsilon=\epsilon^2$
	in Lemma \ref{lem:DOC quadr form Ha-H1+a embedding inequ},
	the following estimate is valid
	\begin{align*}
		2\kappa\sum_{k,j}^{n,k}\theta_{k-j}^{(k)} \myinnerb{f_{\phi}^je^j,-(-\Delta_h)^\alpha e^k}
		\le2\kappa\epsilon^2\sum_{k,j}^{n,k}\theta_{k-j}^{(k)}\myinnerb{(-\Delta_h)^{\frac{1+\alpha}{2}} e^j,(-\Delta_h)^{\frac{1+\alpha}{2}} e^k}
		+\frac{c_{\eta}}{2}
		\sum_{k=2}^n\tau_k  \mynormb{e^k}^2.
	\end{align*}
	Besides, the diffusion term is estimated as
	\begin{align*}
		2\kappa\sum_{k,j}^{n,k}\theta_{k-j}^{(k)} \myinnerb{-\epsilon^2\Delta_he^j,-(-\Delta_h)^\alpha e^k}=-2\kappa\epsilon^2\sum_{k,j}^{n,k}\theta_{k-j}^{(k)} \myinnerb{(-\Delta_h)^{\frac{1+\alpha}{2}}e^j,(-\Delta_h)^{\frac{1+\alpha}{2}} e^k}.
	\end{align*}
	Then the term $Q^n$ in \eqref{Error-quadratic forms} is bounded by
	\begin{align*}
		Q^n\le&\,\frac{c_{\eta}}{2}\sum_{k=2}^n\tau_k  \mynormb{e^k}^2.
	\end{align*}
	
	Therefore, it follows from \eqref{Error-Equation-Inner} that
	\begin{align*}
		\mynormb{e^n}^2
		\le \mynormb{e^1}^2-2\sum_{k=2}^n\theta_{k-2}^{(k)}b_{1}^{(2)}\mynormb{e^k}\mynormb{\diff e^1}
		+\frac{c_{\eta}}{2}\sum_{k=2}^n\tau_k  \mynormb{e^k}^2
		+ 2\sum_{k=2}^n\mynormb{e^k}\mynormb{\Upsilon_{P}^k+\Xi_{\Phi}^k}.
	\end{align*}
	Let integer $n_0$ ($1\le n_0 \le n$) such that
	$\mynormb{e^{n_0}}=\max_{1\le k \le n}\mynormb{e^k}$.
	Taking $n:=n_0$ in the above inequality results in 
	\begin{align*}
		\mynormb{e^{n_0}}\le \mynormb{e^1}-2\mynormb{\partial_{\tau} e^1}\sum_{k=2}^{n_0}\theta_{k-2}^{(k)}b_{1}^{(2)}\tau_1
		+\frac{c_{\eta}}{2}\sum_{k=2}^{n_0}\tau_k  \mynormb{e^k}+ 2\sum_{k=2}^{n_0}
		\mynormb{\Upsilon_{P}^k+\Xi_{\Phi}^k}.
	\end{align*}
	By using Lemma \ref{lem: DOC property}, we get 
	$$-\sum_{k=2}^{n}\theta_{k-2}^{(k)}b_{1}^{(2)}\tau_1=\sum_{k=2}^{n}\tau_1\prod_{i=2}^k\frac{r_i^2}{1+2r_i}=\sum_{k=2}^{n}\tau_k\prod_{i=2}^k\frac{r_i}{1+2r_i}\le\sum_{k=2}^{n}\frac{\tau_k}{2^{k-1}}\le \tau.$$
Subsequently, we arrive at
	\begin{align*}
	\mynormb{e^n}\le\mynormb{e^{n_0}}
		&\le \mynormb{e^1}+2\tau\mynormb{\partial_{\tau} e^1}
		+\frac{c_{\eta}}{2}\sum_{k=2}^{n}\tau_k  \mynormb{e^k}
		+ 2\sum_{k=2}^{n}\mynormb{\Upsilon_{P}^k+\Xi_{\Phi}^k}.
	\end{align*}
	Under the maximum step constraint $\tau\le 1/c_{\eta}$, we have
	\begin{align*}
		\mynormb{e^n}\le 2\mynormb{e^1}+4\tau\mynormb{\partial_{\tau} e^1}
		+c_{\eta}\sum_{k=2}^{n-1}\tau_k  \mynormb{e^k}
		+ 4\sum_{k=2}^n\mynormb{\Upsilon_{P}^k+\Xi_{\Phi}^k}.
	\end{align*}
	The discrete Gr\"onwall inequality \cite[Lemma 3.1]{LiaoZhang:2020linear} yields the following estimate
	\begin{align*}
		\mynormb{e^n}&\le 2\exp\brab{c_{\eta} t_{n-1}}\braB{\mynormb{e^1}+2\tau\mynormb{\partial_{\tau} e^1}
			+ 2\sum_{k=2}^n\mynormb{\Upsilon_{P}^k}+ 2\sum_{k=2}^n\mynormb{\Xi_{\Phi}^k}}\quad\text{for $2\le n\le N$.}
	\end{align*}
	Finally, by applying \eqref{Projection-consistency}, \eqref{temporal error} and the triangle inequality \eqref{Triangle-Projection-Estimate}, the proof is completed.
\end{proof}

\section{Numerical experiments}
In this section, two numerical experiments with different initial conditions are performed to confirm the theoretical results presented in previous sections. The first one aims at testing the convergence rates for the time discretization of the variable-step BDF2 scheme\eqref{scheme:FCH BDF2}. The second one is provided to investigate the effectiveness of the adaptive time-stepping algorithm. Moreover, we simulate the process of energy decay for the numerical solution with different parameters. The fixed point iteration is applied to solve the nonlinear equations at each time level with the termination error $10^{-12}$.

\begin{example}\label{convergence}
	We test the convergence results stated in Theorem \ref{thm: L2 Convergence-CH}. In order to check the convergence order, we consider the 2D space FCH equation\eqref{cont: Problem-FCH} {on the domain $\Omega=(0,2\pi)^{2}$} with the initial condition $\Phi_0(\mathbf{x})= \sin x \sin y$. 
\end{example}

This problem was used by Bu et al. \cite{Bu2020Appl.Numer.Math} to verify the convergence rates. It was proved that the convergence rates were verified on the uniform time meshes with the above initial condition. 
The current example is solved on the random time meshes. In more details, let the time step sizes $\tau_k:=T\sigma_{k}/S$ for $1\leq k\leq N$,
where $\sigma_{k}\in(0,1)$ is the uniformly distributed random number
and $S=\sum_{k=1}^N\sigma_{k}$. We denote the convergence order in the following
$$\text{Order}:=\frac{\log\bra{e(N)/e(2N)}}{\log\bra{\tau(N)/\tau(2N)}},$$
where  $e(N)$ corresponding to the discrete $L^2$ norms of the error between the
numerical solution and the reference solution, and $\tau(N)$ denotes the maximum time-step size for total $N$ subintervals.
\begin{table}[htb!]
	\begin{center}
		\caption{$L^2$ norm errors and convergence rates with {$\epsilon=1/\sqrt{10}$ and $\kappa=1$}. }\label{examp:FCH BDF2 L2 error} 
		\vspace*{0.3pt}
		\def\temptablewidth{0.80\textwidth}
		{\rule{\temptablewidth}{0.5pt}}
		\begin{tabular*}{\temptablewidth}{@{\extracolsep{\fill}}ccccccc}
			\multirow{2}{*}{$N$}  & \multicolumn{3}{r}{$\alpha$=0.05} & \multicolumn{3}{r}{$\alpha$=0.4}\\ \cline{2-4}\cline{5-7}
			\multicolumn{1}{c}{} & $e(N)$ & Order &$\max r_k$ & $e(N)$ & Order &$\max r_k$\\ \hline
			32   &1.68e-04 &2.24 &78.19 &1.46e-04 &2.08 &5.79\\
			64   &3.57e-05 &2.05 &20.20 &3.87e-05 &2.13 &83.72\\
			128  &8.76e-06 &2.17 &52.74 &9.19e-06 &2.02 &30.40\\
			256  &2.13e-06 &2.05 &385.45 &2.44e-06 &2.05  &126.44\\ 
			512  &5.42e-07 &1.87 &1408.48 &6.34e-07 &1.96 &2112.31\\			
			\hline
			\multirow{2}{*}{$N$}  & \multicolumn{3}{r}{$\alpha$=0.6} & \multicolumn{3}{r}{$\alpha$=0.95}\\ \cline{2-4}\cline{5-7}
			\multicolumn{1}{c}{} & $e(N)$ & Order &$\max r_k$ & $e(N)$ & Order &$\max r_k$\\ \hline
			32   &1.20e-04 &1.79 &4.96 &1.42e-04 &1.97 &321.97\\
			64   &3.11e-05 &2.34 &88.97 &2.10e-05 &2.06 &86.95\\
			128  &8.62e-06 &1.85 &223.57 &5.52e-06 &2.09 &542.23\\
			256  &2.14e-06 &2.17 &181.97 &1.15e-06 &2.11 &39.29\\ 
			512  &5.54e-07 &1.92 &850.80 &2.50e-07 &2.29 &269.41\\
		\end{tabular*}
		{\rule{\temptablewidth}{0.5pt}}
	\end{center}
\end{table}

\begin{table}[htb!]
	\begin{center}
		\caption{$L^2$ norm errors and convergence rates with {$\epsilon=0.05$ and $\kappa=0.1$}. }\label{examp:FCH BDF2 L2 error2} 
		\vspace*{0.3pt}
		\def\temptablewidth{0.80\textwidth}
		{\rule{\temptablewidth}{0.5pt}}
		\begin{tabular*}{\temptablewidth}{@{\extracolsep{\fill}}ccccccc}
			\multirow{2}{*}{$N$}  & \multicolumn{3}{r}{$\alpha$=0.05} & \multicolumn{3}{r}{$\alpha$=0.4}\\ \cline{2-4}\cline{5-7}
			\multicolumn{1}{c}{} & $e(N)$ & Order &$\max r_k$ & $e(N)$ & Order &$\max r_k$\\ \hline
			32   &8.20e-07 &1.80 &44.17 &1.23e-05 &1.94 &25.49\\
			64   &2.06e-07 &1.91 &37.02 &2.78e-06 &1.78	&417.24\\
			128  &5.78e-08 &2.12 &158.82 &7.22e-07 &2.35 &58.93\\
			256  &1.36e-08 &2.00 &116.05 &1.75e-07 &1.85 &174.13\\ 
			512  &3.56e-09 &1.94 &700.72 &4.42e-08 &2.03 &853.16\\
			
			\hline
			\multirow{2}{*}{$N$}  & \multicolumn{3}{r}{$\alpha$=0.6} & \multicolumn{3}{r}{$\alpha$=0.95}\\ \cline{2-4}\cline{5-7}
			\multicolumn{1}{c}{} & $e(N)$ & Order &$\max r_k$ & $e(N)$ & Order &$\max r_k$\\ \hline
			32   &6.62e-05 &1.70 &30.75 &2.54e-02 &2.24 &26.79\\
			64   &2.14e-05 &2.20 &76.78 &6.34e-03 &1.83 &20.71\\
			128  &4.93e-06 &2.06 &58.57 &1.71e-03 &2.38 &166.45\\
			256  &1.26e-06 &1.86 &337.16 &4.77e-04 &1.90 &119.31\\ 
			512  &3.21e-07 &2.22 &485.09 &1.09e-04 &1.93 &134.05\\
		\end{tabular*}
		{\rule{\temptablewidth}{0.5pt}}
	\end{center}
\end{table}

The spatial grid size used in the calculations is $M = 128$. 
Since the exact solution for the example is unknown, we use numerical results of uniform BDF2 scheme with $\tau_n = 10^{-4}$ and $M = 128$ as the reference solution. We begin with $\epsilon=1/\sqrt{10}$ and $\kappa=1$. The fractional order $\alpha$ is set to be $0.05, 0.4, 0.6, 0.95$ and the numerical errors are computed at $T = 1$. Table \ref{examp:FCH BDF2 L2 error} shows the $L^2$ errors, the convergence rates and the maximum time-step ratios $r_k$ corresponding to the different values of $N$. It is seen that the expected second order convergence rate in time is obtained on arbitrary nonuniform meshes. Furthermore, it is observed that the convergence orders are achieved as the maximum time-step ratios $r_k$ is much higher than the step ratios limit $4.864$ as claimed in Theorem \ref{thm: L2 Convergence-CH}.  
Besides, we test the convergence rates for the space FCH equation with a smaller $\epsilon= 0.05$ and $\kappa=0.1$ at $T=1$. The results displayed in Table \ref{examp:FCH BDF2 L2 error2} also shows the second-order accuracy in time direction.


%
%


\begin{figure}[htb!]
	\centering
	\begin{minipage}[c]{0.45\textwidth}
		\centering
		\includegraphics[width=8cm]{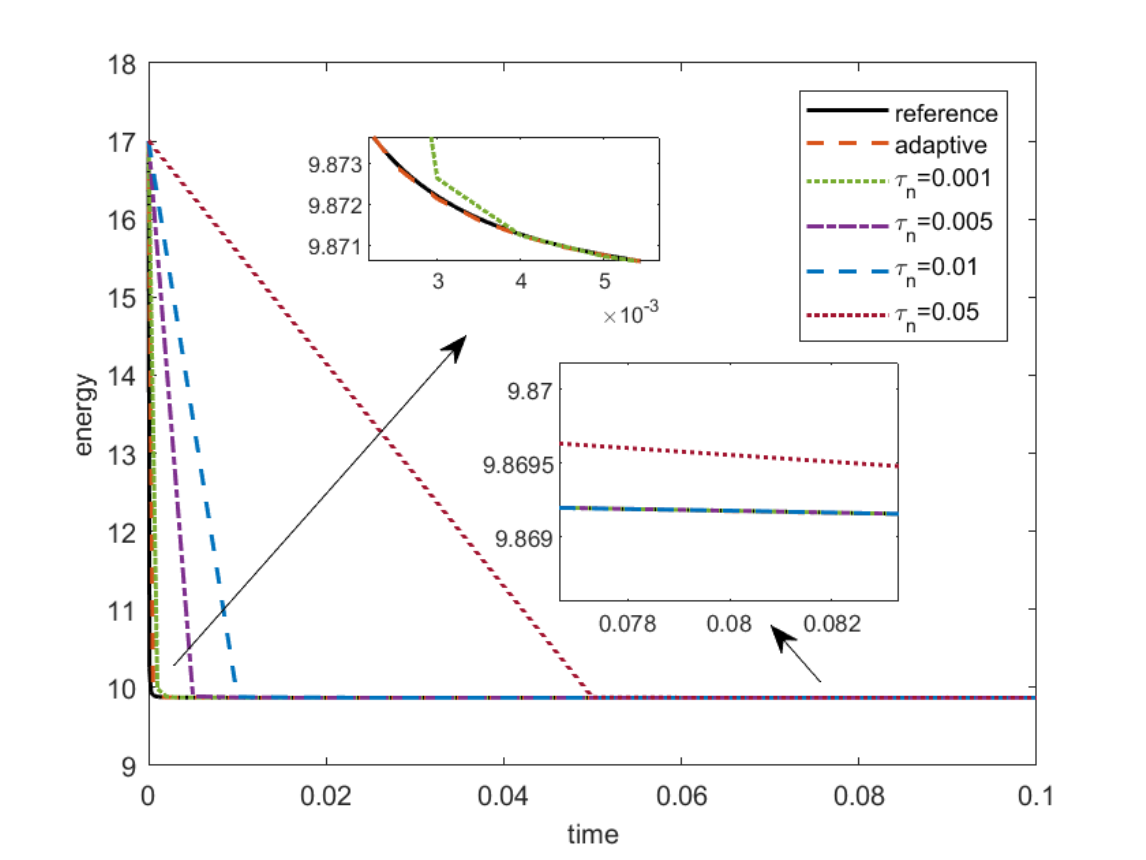}
	\end{minipage}
	\begin{minipage}[c]{0.45\textwidth}
		\centering
		\includegraphics[width=8cm]{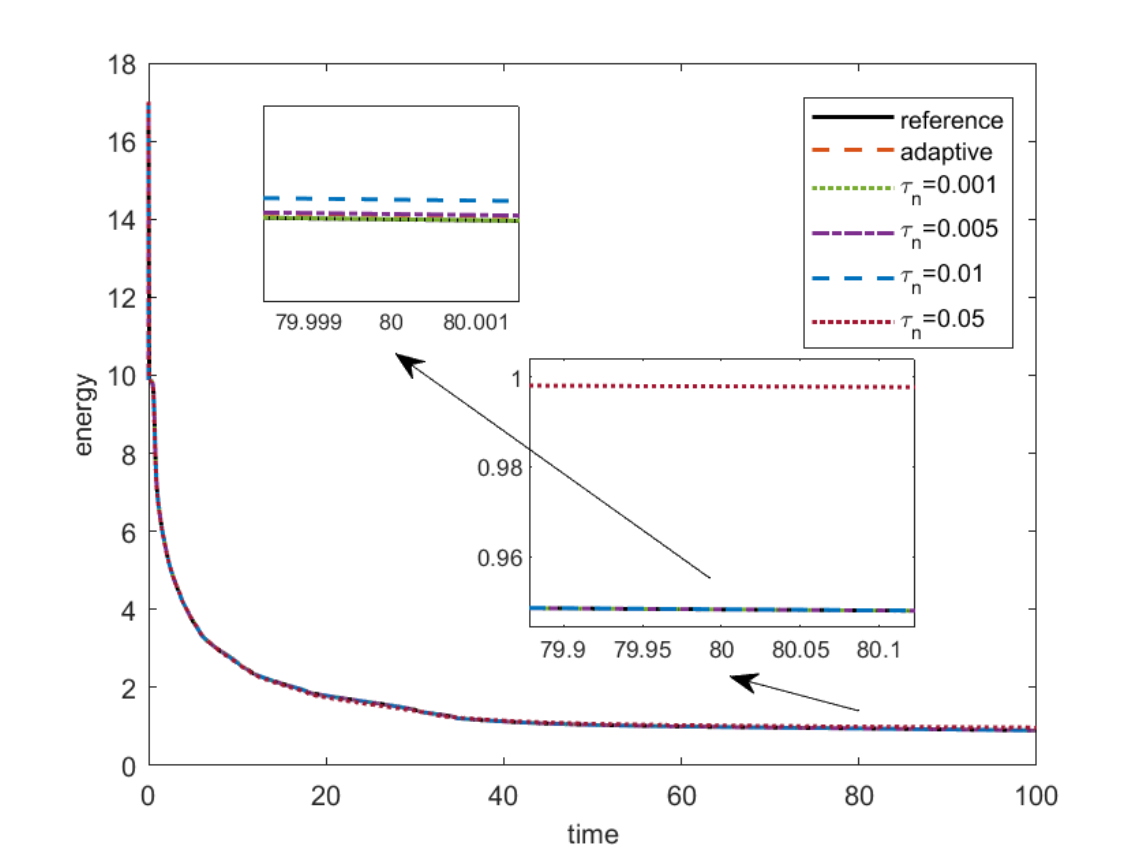}
	\end{minipage}
	\vspace{8pt}
	\caption{Evolutions of energy obtained by adaptive time steps and uniform time steps till $T=0.1$ (a) and $T=100$ (b).}
	\label{efficent2}
\end{figure}

\begin{example}\label{efficency}
	To demonstrate the effectiveness of the implicit variable-step BDF2 scheme \eqref{scheme:FCH BDF2}, we stimulate the space FCH model with a uniformly random distribution field in $[-0.1,0.1]$ taken as the initial data.
\end{example}

We consider the space FCH model on the domain $\Omega=(0,2\pi)^{2}$. By discretizing the domain using $N = 512\times 512$ mesh points, we solved the space FCH equation to final time $T=100$ adopting the adaptive time-stepping strategy and uniform time steps, respectively. The adaptive time-stepping strategy \cite{HuangYangWei:2020} is given by
\begin{align}\label{algo:adaptive step}
	\tau_{ada}
	=\max\Bigg\{\tau_{\min},
	\frac{\tau_{\max}}{\sqrt{1+\rho\mynormb{\partial_\tau \phi^n}^2}}\Bigg\}
	\quad\text{such that}\quad
	\tau_{n+1}=\min\big\{\tau_{ada},r_{\text{user}}\tau_n\big\},
\end{align}
where $\tau_{\max}$, $\tau_{\min}$ are preset maximum and minimum time steps, and $\rho$ is a positive prechosen parameter related to the level of adaptivity. The adaptive time steps are selected based on the numerical variation. In this case, we take $\rho=10^3$, $r_{\text{user}}=4$, $\tau_{\min}=5\times10^{-4}$ and $\tau_{\max}=10^{-1}$. The reference solution is obtained with $\tau_n=10^{-4}$ using the implicit BDF2 scheme.

\begin{table}[htb!]
	\begin{center}
		\caption{Comparisons of CPU time (in seconds) and total time levels.}\label{CPU time2} \vspace*{0.3pt}
		\def\temptablewidth{0.8\textwidth}
		{\rule{\temptablewidth}{0.5pt}}
		\begin{tabular*}{\temptablewidth}{@{\extracolsep{\fill}}ccccc}
			& &adaptive time steps  &$\tau=0.001$  &$\tau=0.005$   \\
			\midrule
			&CPU time &2.6566e+04 &1.2951e+05	 &5.7906e+04	 	\\	
			&Time levels&13611 &100000&20000\\	
		\end{tabular*}
		{\rule{\temptablewidth}{0.5pt}}
	\end{center}
\end{table}	
\begin{figure}[htb!]
	\centering
	\begin{minipage}[c]{0.45\textwidth}
		\centering
		\includegraphics[width=8cm]{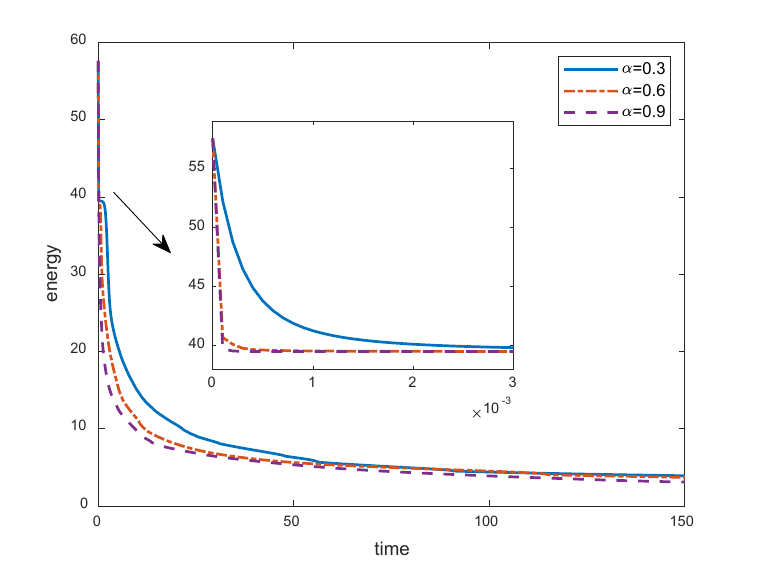}
	\end{minipage}
	\begin{minipage}[c]{0.45\textwidth}
		\centering
		\includegraphics[width=8cm]{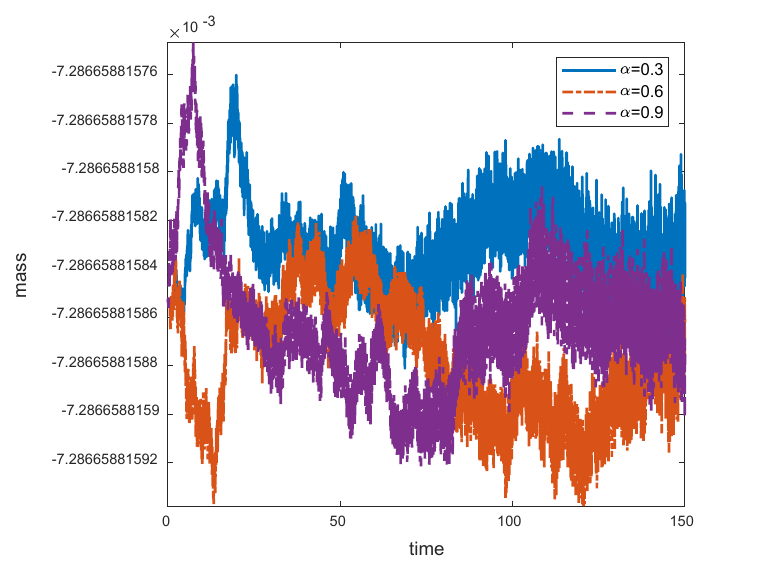}
	\end{minipage}
	\vspace{8pt}
	\caption{Evolutions of energy (a) and mass (b) for different values of $\alpha$.}
	\label{examp:energy stable2}
\end{figure}

Figure \ref{efficent2} shows evolutions of energy with $\alpha=0.5$,  $\epsilon=0.05$ and $\kappa=1$ computed utilizing fixed time steps $\tau_n=0.05, 0.01, 0.005,0.001$ and the adaptive time steps. It is obvious that the adaptive time algorithm captures the quick variation of the energy successfully, while the uniform time steps with $\tau_n=0.05,0.01,0.005$ are failed to detect the initial variations. In fact, only 29 time levels are computed in the simulation with adaptive time-stepping strategy, while that for the fixed time steps $\tau_n=0.001$ is 100.  
Meanwhile, Figure \ref{efficent2}(b) presents that the energy curves computed by the adaptive time-stepping strategy and uniform steps $\tau_n=0.005,0.001$ are almost compatible with the one using the uniform step $\tau_n=10^{-4}$ in the long time simulation. 

\begin{figure}[htb!]
	\begin{minipage}{0.25\linewidth} 	
		\centerline{\includegraphics[width=3.5cm]{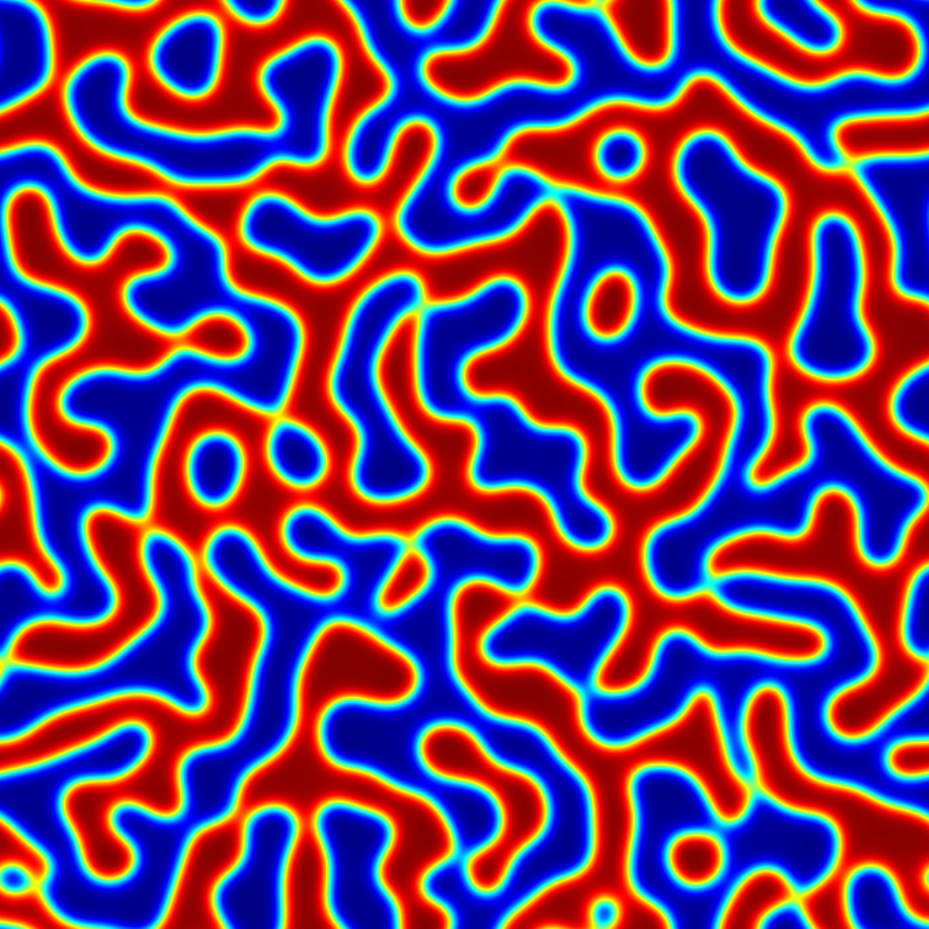}}
		\vspace{4pt}
		\centerline{\includegraphics[width=3.5cm]{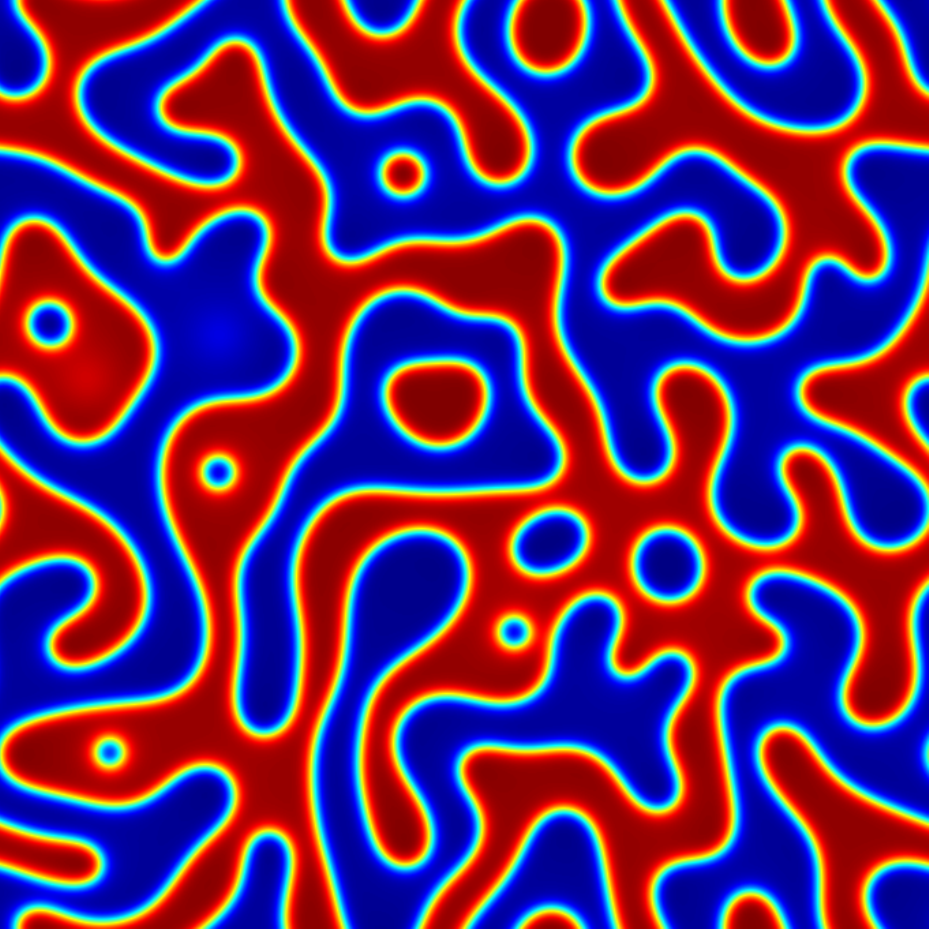}}
		\vspace{4pt}
		\centerline{\includegraphics[width=3.5cm]{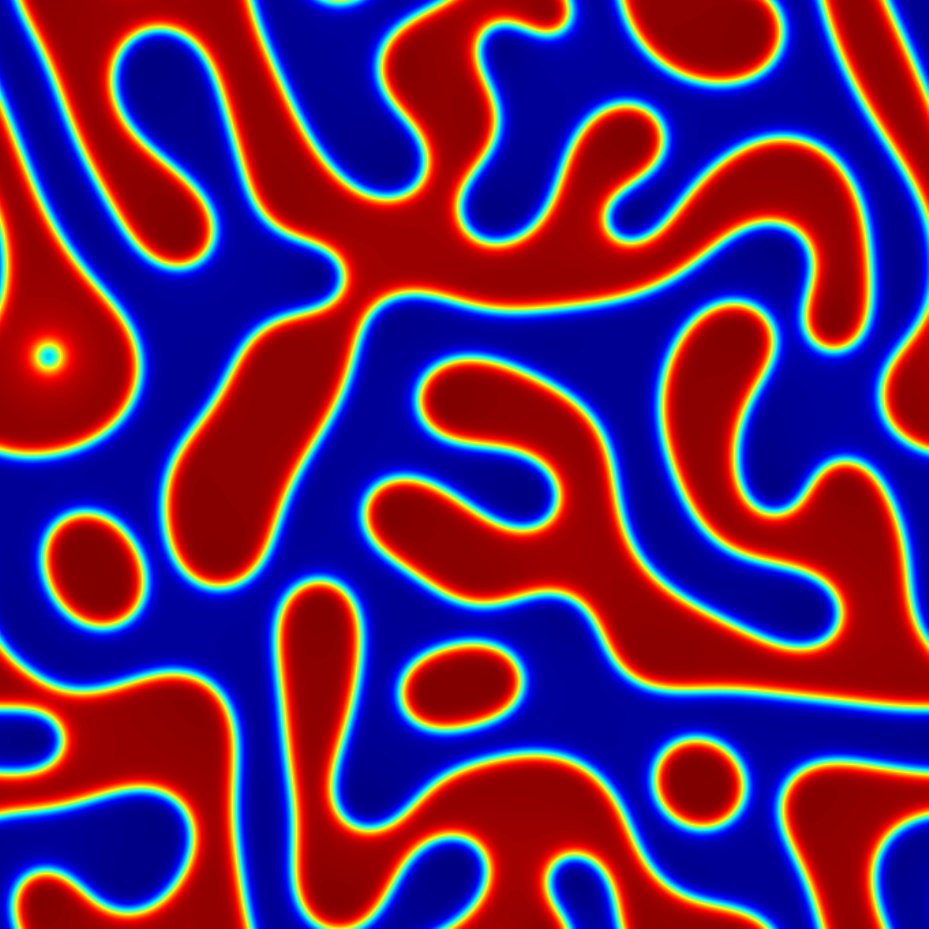}}
	\end{minipage}
	\hspace{-3mm}
	\begin{minipage}{0.25\linewidth}
		\centerline{\includegraphics[width=3.5cm]{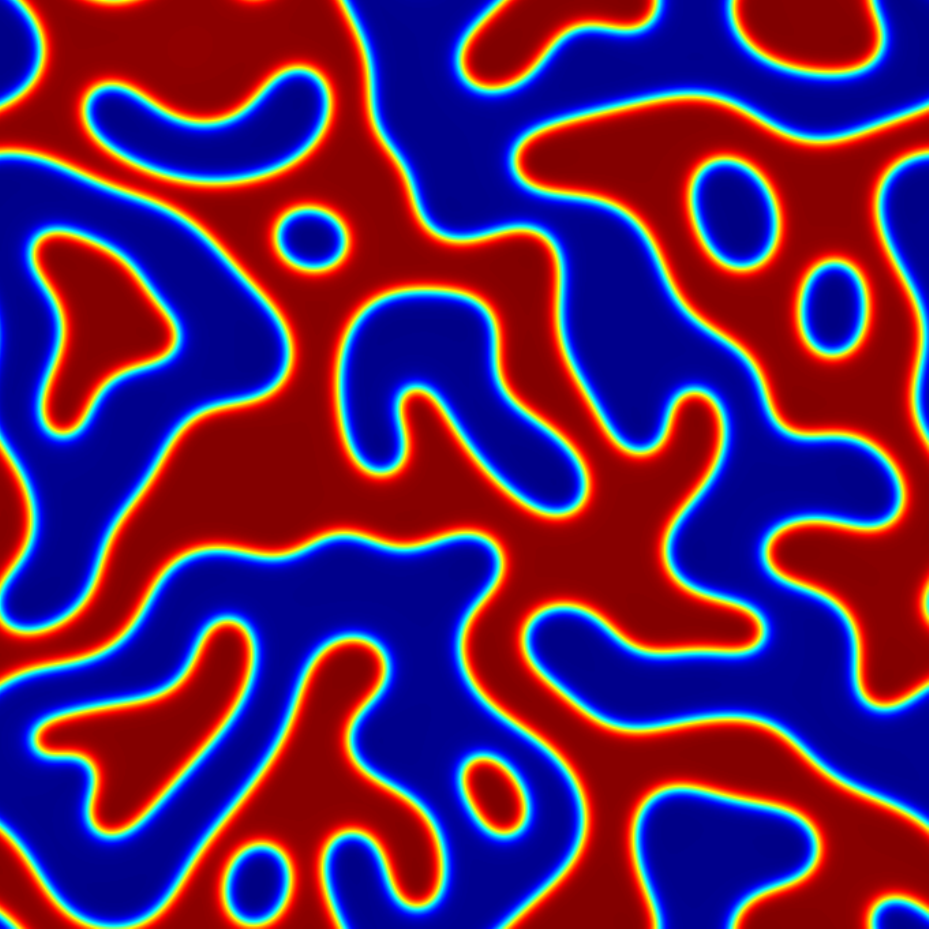}}
		\vspace{4pt}
		\centerline{\includegraphics[width=3.5cm]{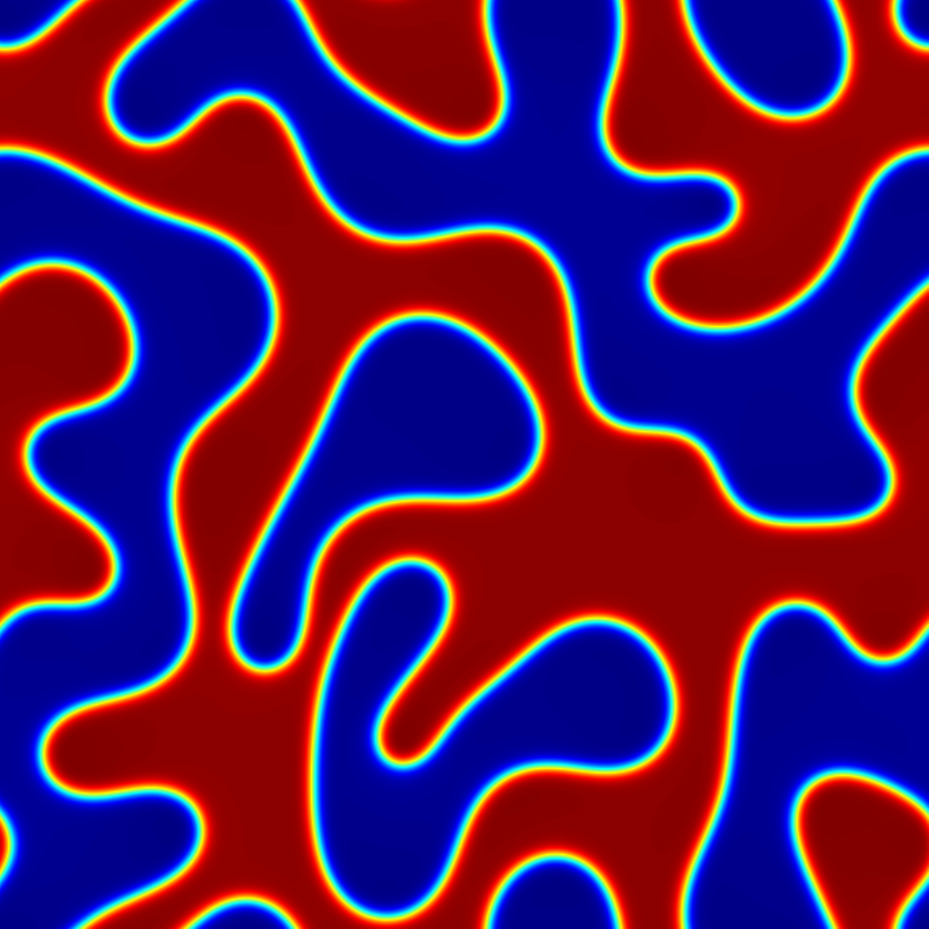}}
		\vspace{4pt}
		\centerline{\includegraphics[width=3.5cm]{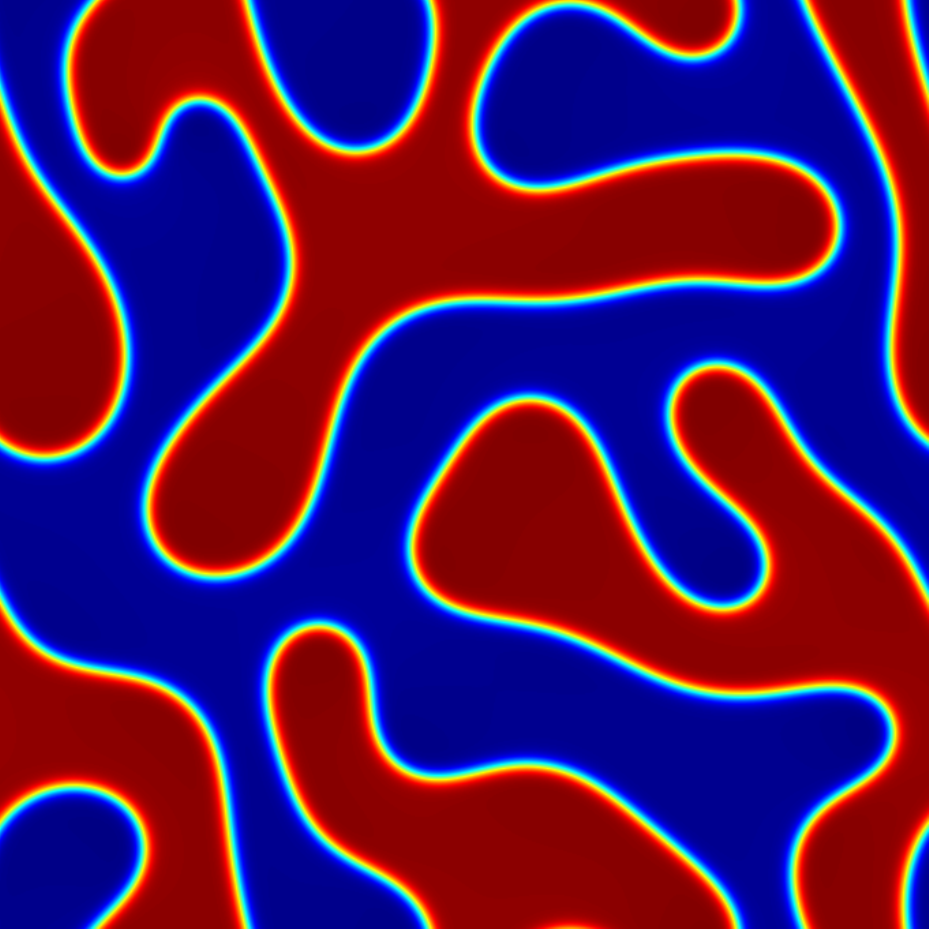}}
	\end{minipage}
	\hspace{-3mm}
	\begin{minipage}{0.25\linewidth}
		\centerline{\includegraphics[width=3.5cm]{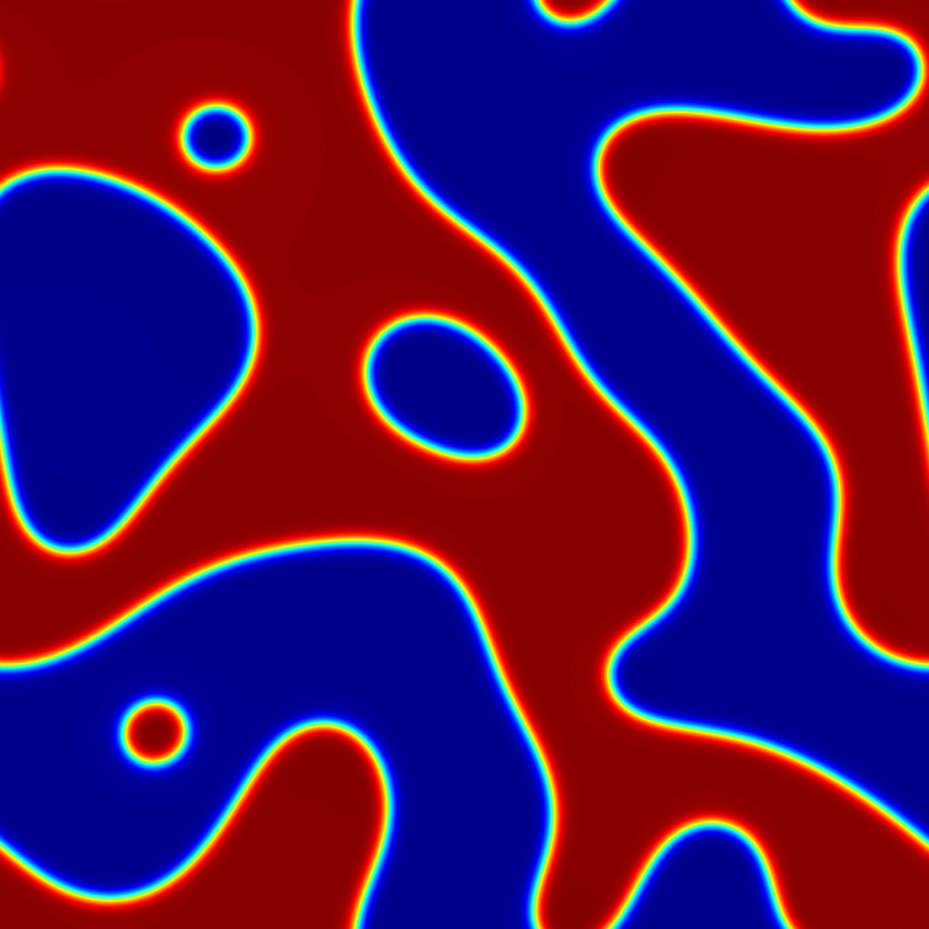}}
		\vspace{4pt}
		\centerline{\includegraphics[width=3.5cm]{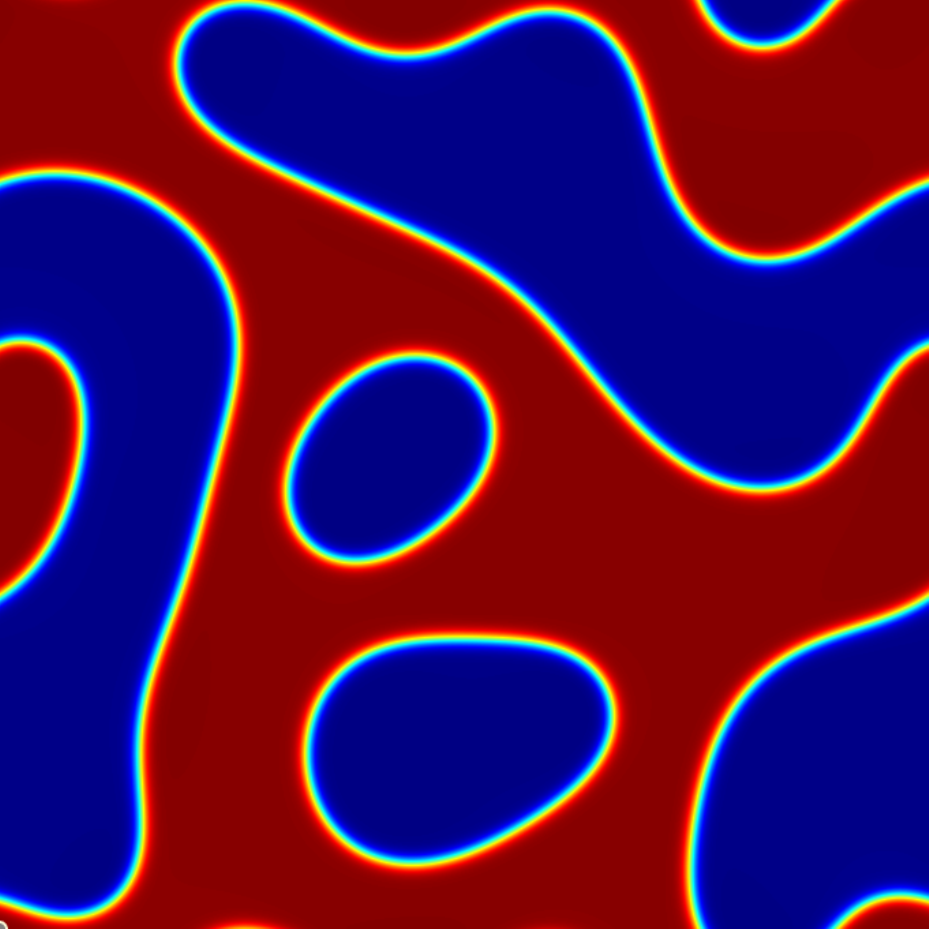}}
		\vspace{4pt}
		\centerline{\includegraphics[width=3.5cm]{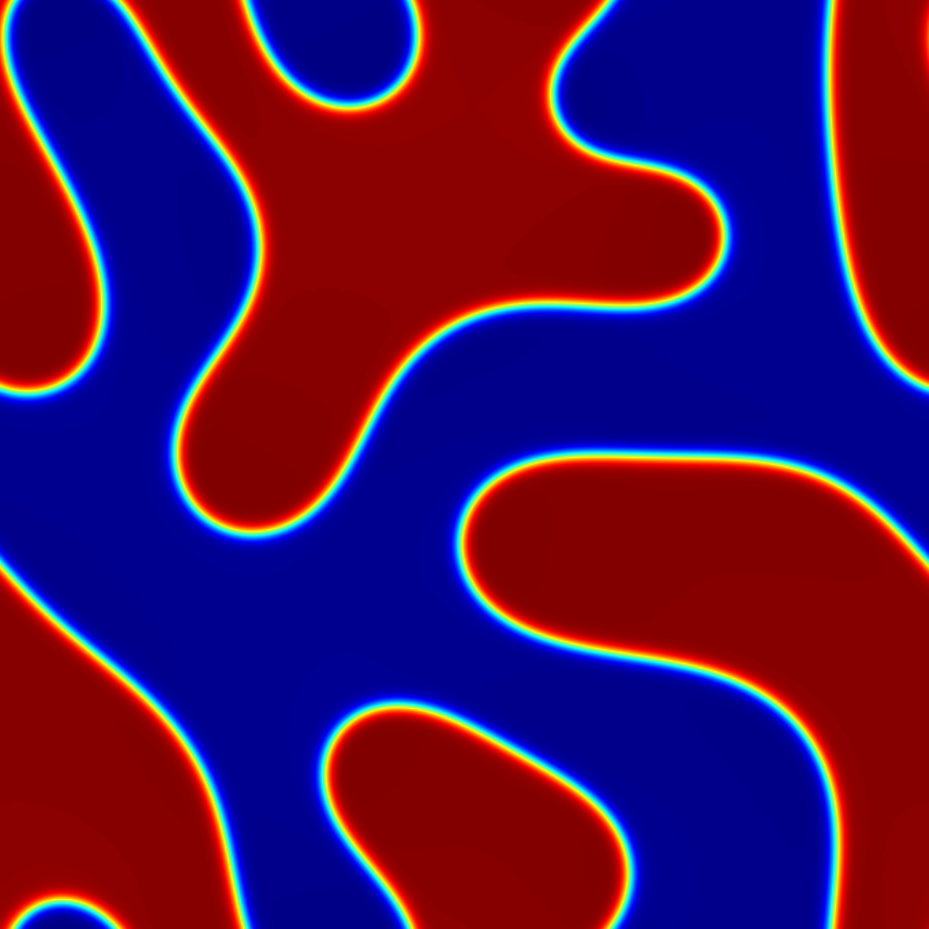}}
	\end{minipage}
	\hspace{-3mm}
	\begin{minipage}{0.25\linewidth}
		\centerline{\includegraphics[width=3.5cm]{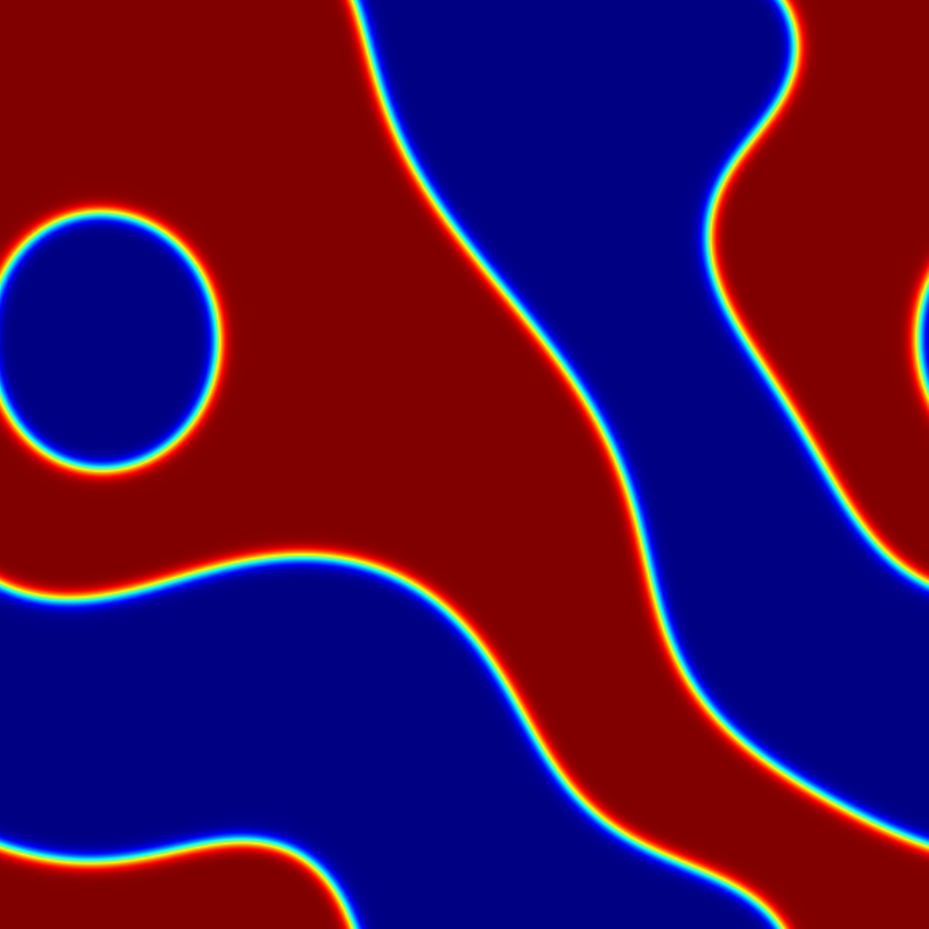}}
		\vspace{4pt}
		\centerline{\includegraphics[width=3.5cm]{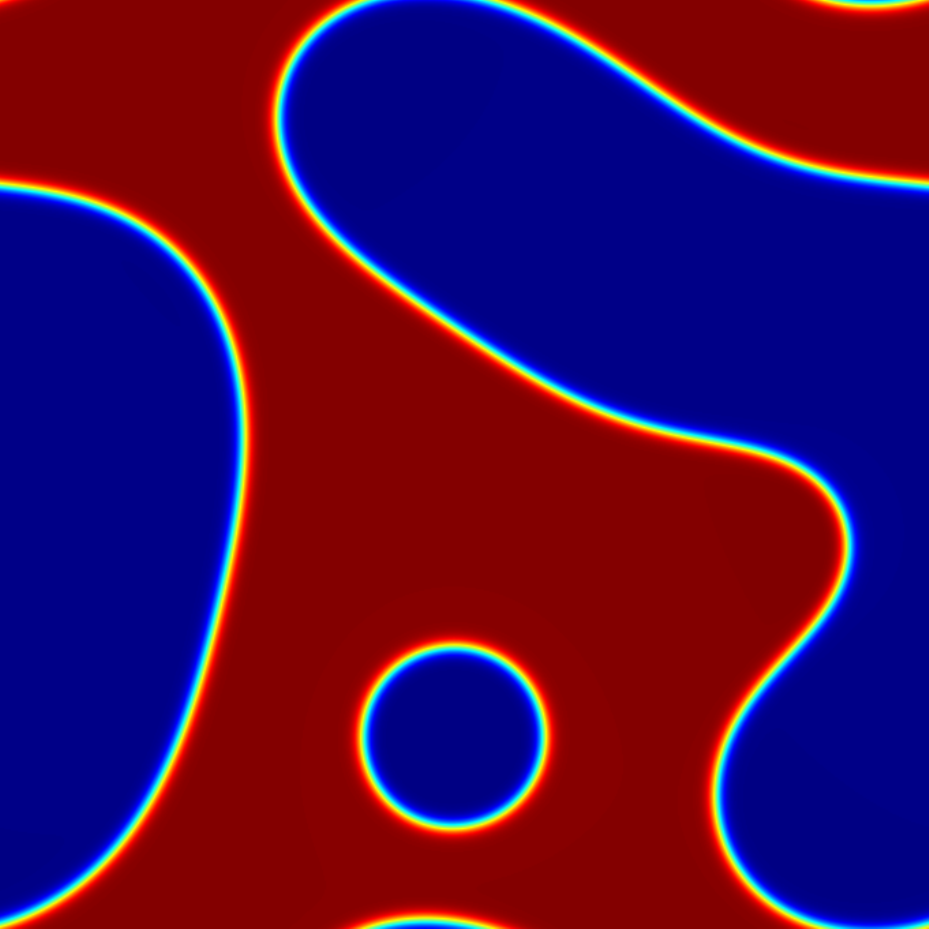}}
		\vspace{4pt}
		\centerline{\includegraphics[width=3.5cm]{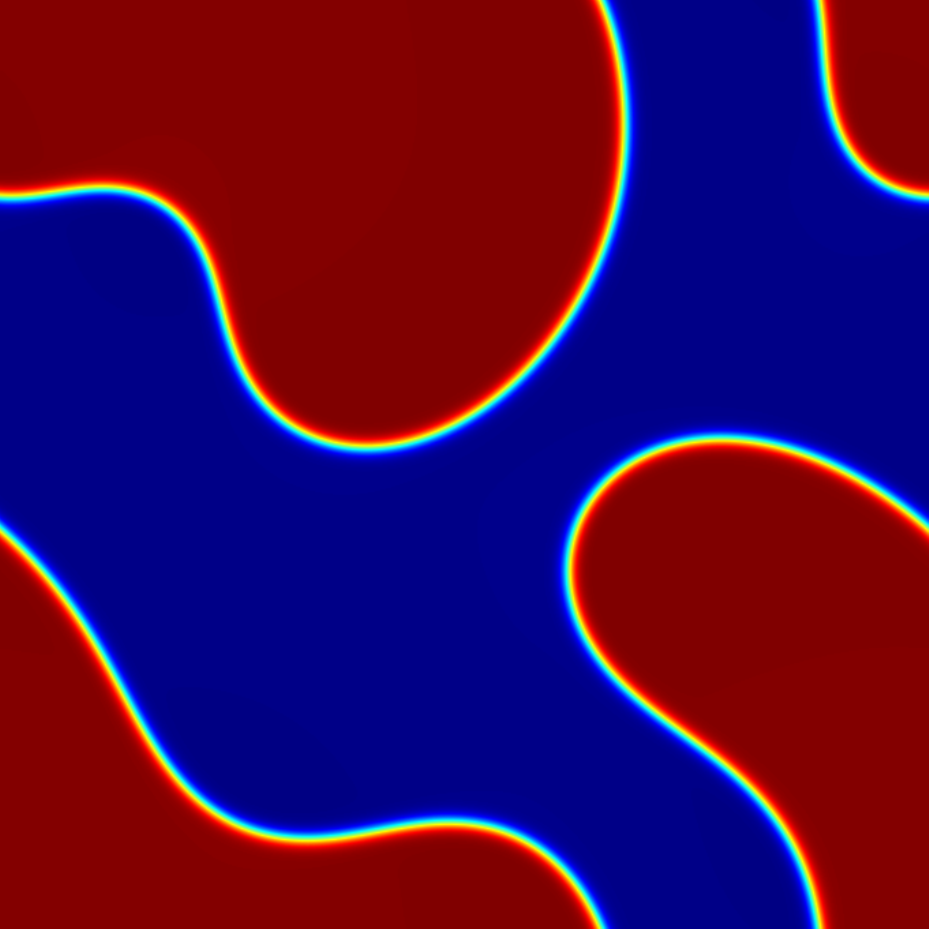}}
	\end{minipage}
	\vspace{8pt}
	\caption{Snapshots of solutions of the space FCH model at different time $T = 5, 15, 50, 150$ (from left column to right one) with $\alpha= 0.3,0.6,0.9$ (from top to bottom).}
	\label{snapshots}
\end{figure}

{Moreover, Table \ref{CPU time2} lists the comparisons of  CPU time and required time levels. It shows that the adaptive time-stepping strategy reduces the CPU time significantly for long time simulation, especially for the space fractional problem, and the total number of time steps is reduced eventually.
More precisely, 
the adaptive time-stepping strategy saves the CPU time by approximately 55\% compared to the case with the fixed time step $\tau_n=0.005$, and by about 80\% compared to the one utilizing the uniform time step $\tau_n=0.001$.}

\begin{figure}[htb!]
	\centering
	\begin{minipage}[c]{0.45\textwidth}
		\centering
		\includegraphics[width=8cm]{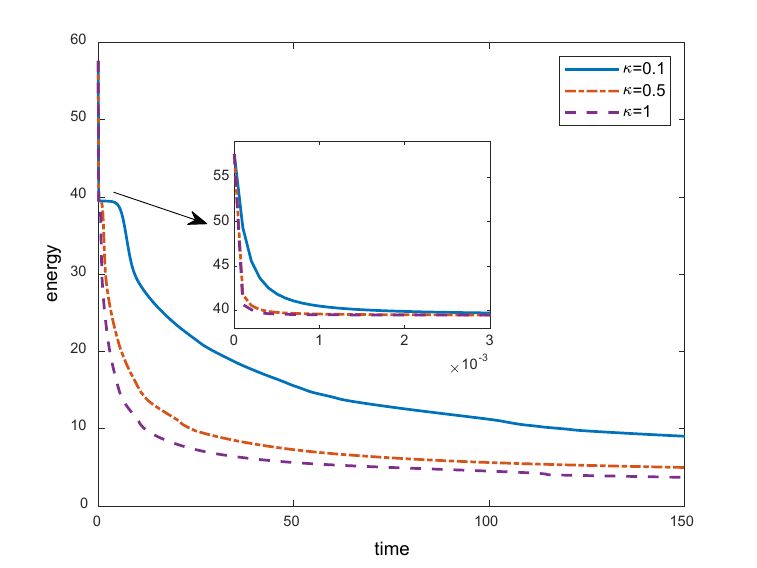}
	\end{minipage}
	\begin{minipage}[c]{0.45\textwidth}
		\centering
		\includegraphics[width=8cm]{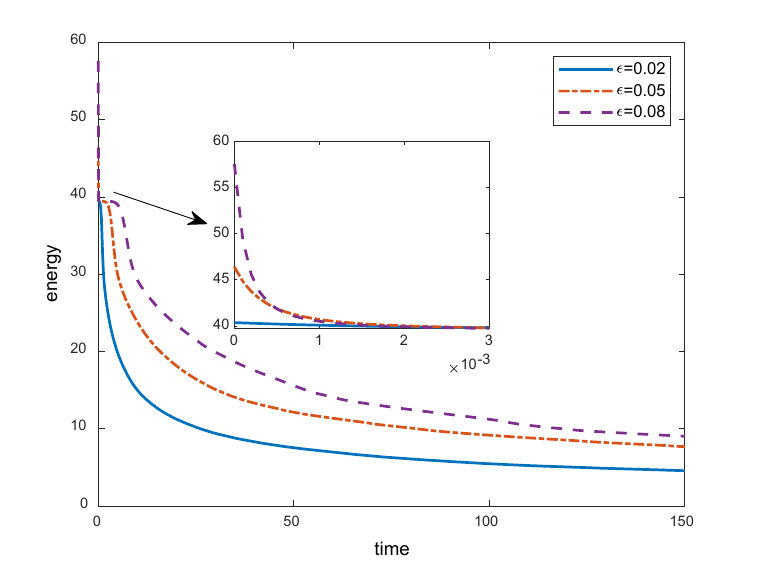}
	\end{minipage}
	\vspace{8pt}
	\caption{Evolutions of energy for different values of $\kappa$ (a) and $\epsilon$ (b).}
	\label{comparekappa}
\end{figure}


{In the following, we study the mass conservation and energy dissipation law of the problem with  different model parameters on $\Omega=(-2\pi,2\pi)^{2}$ to final time $T=150$.
We choose $\rho=10^{3}$, $r_{\text{user}}=4$, $\tau_{\min}=10^{-4}$ and $\tau_{\max}=5\times10^{-2}$ in the adaptive time-stepping strategy \eqref{algo:adaptive step}.
The decay of energy for the short time $T=3\times10^{-3}$ and the relatively steady state $T=150$ with different fractional orders $\alpha=0.3,0.6,0.9$,  $\epsilon=0.08$ and $\kappa=1$ using adaptive time steps are shown in Figure \ref{examp:energy stable2}(a). It is clearly indicate that the adaptive time-stepping strategy could capture the multiple time scales behaviors since small time steps are chosen when the solution has a quick transition. The rate at which the solution approaches the steady asymptotic state varies with the fractional order $\alpha$ \cite{Mark2017SIAM}. 
To be more specific, the transition decays slower from the initial state to the steady state for smaller value of fractional order. 

\begin{figure}[htb!]
	\begin{minipage}{0.33\linewidth} 	
		\centerline{\includegraphics[width=4.5cm]{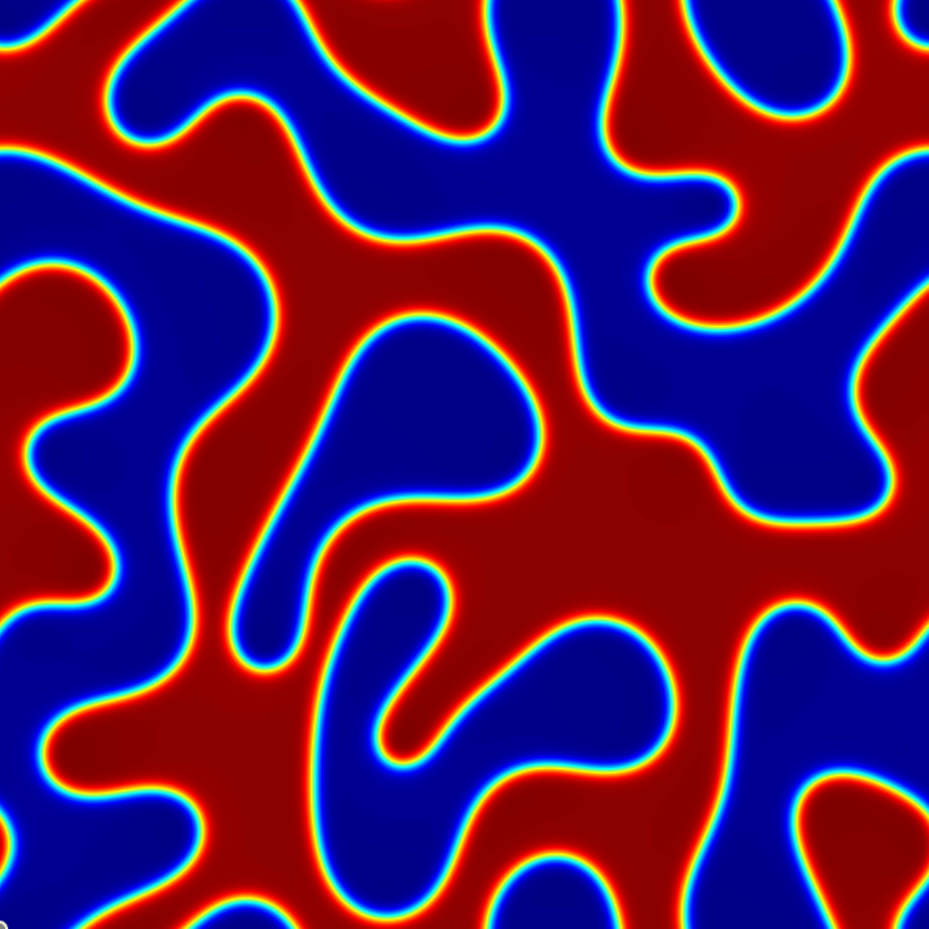}}
	\end{minipage}
	\hspace{-3mm}
	\begin{minipage}{0.33\linewidth}
		\centerline{\includegraphics[width=4.5cm]{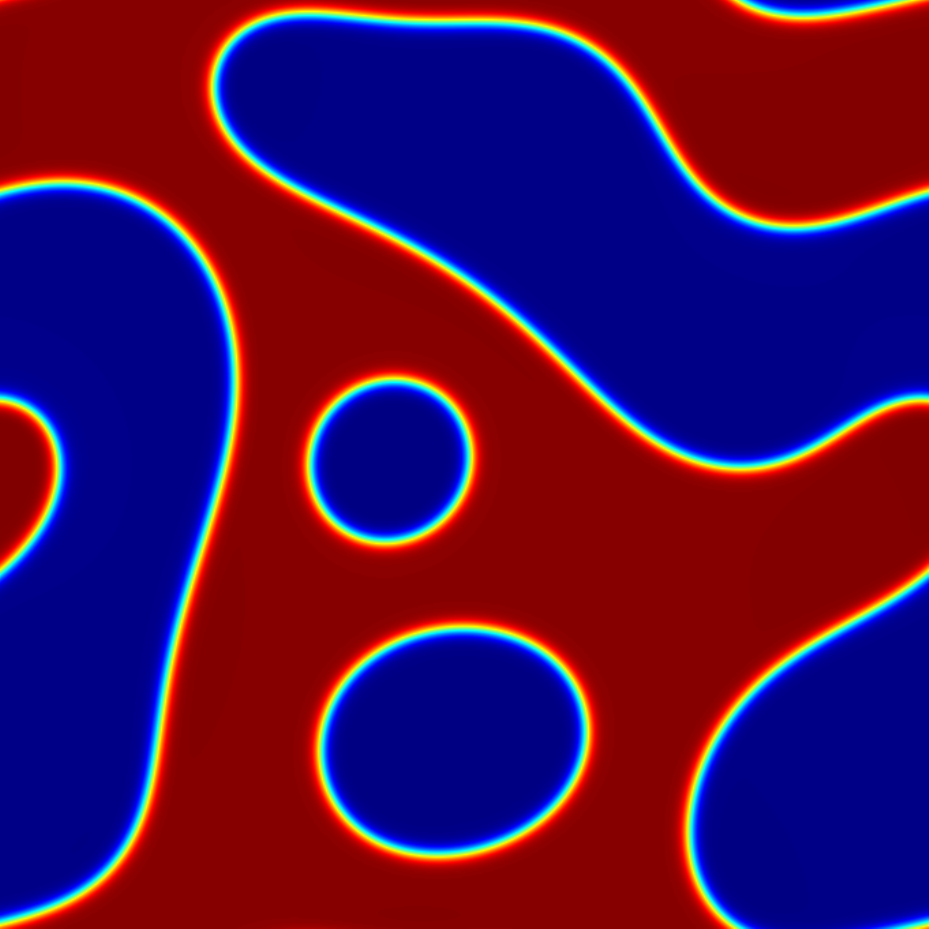}}
	\end{minipage}
	\hspace{-3mm}
	\begin{minipage}{0.33\linewidth}
		\centerline{\includegraphics[width=4.5cm]{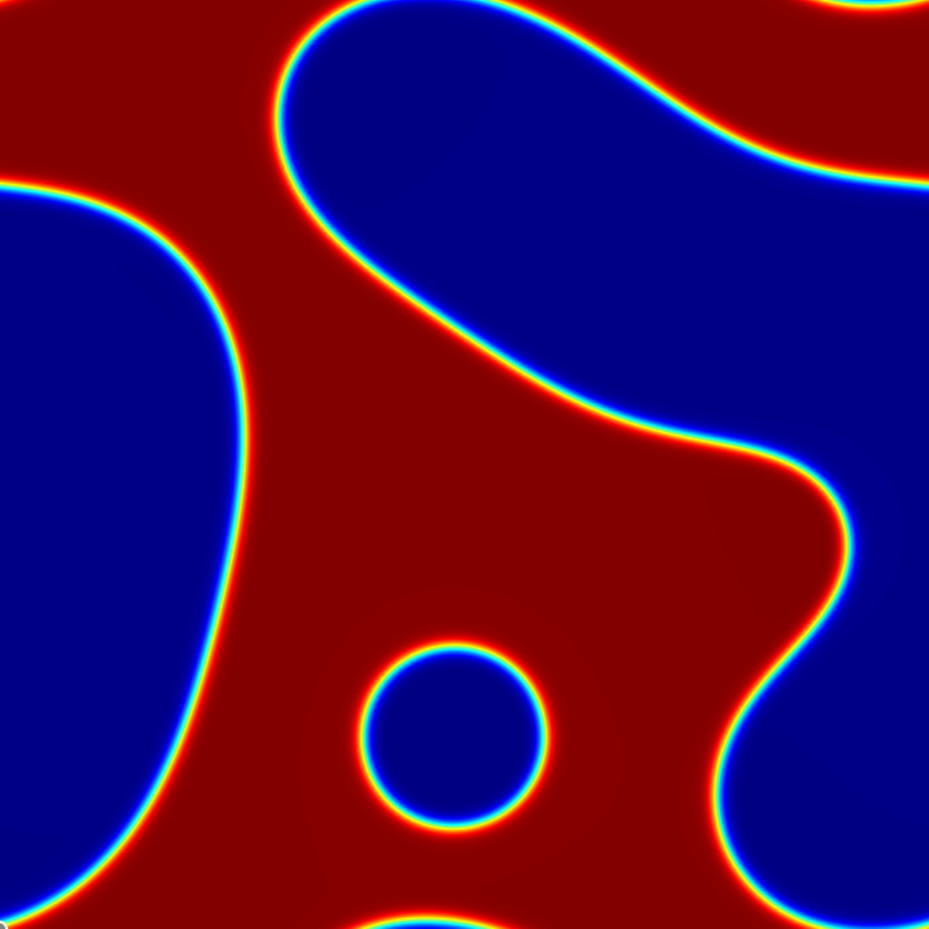}}
	\end{minipage}
	\vspace{8pt}
	\caption{The approximate solutions with $\epsilon=0.08$ and $\kappa=0.1,0.5,1$ at time $T = 150$.}
	\label{snapshotskappa}
\end{figure}

Meanwhile, evolutions of mass obtained by the implicit variable-step BDF2 scheme \eqref{scheme:FCH BDF2} is demonstrated in Figure \ref{examp:energy stable2}(b). As expected, we observe that the developed scheme conserves mass  during the entire dynamical process. 
Moreover, snapshots of numerical solutions for $\alpha=0.3, 0.6, 0.9$, $\epsilon=0.08$ and $\kappa=1$ at time $T=5,15,50,150$ are presented in Figure \ref{snapshots}.
Figure \ref{snapshots} shows that the time adaptive strategy captures the dynamical solution accurately. Besides, the solutions evolve faster for larger order of the fractional Laplacian which is in good agreement with our previous observation.

\begin{figure}[htb!]
	\begin{minipage}{0.33\linewidth} 	
		\centerline{\includegraphics[width=4.5cm]{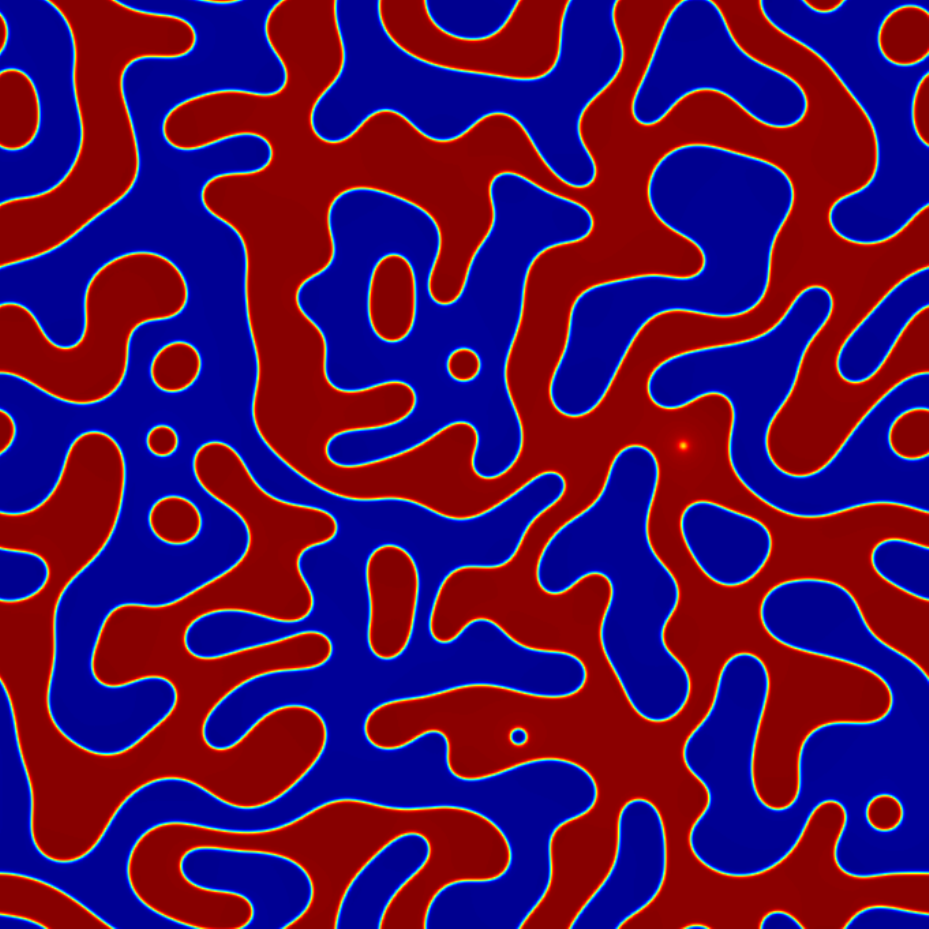}}
	\end{minipage}
	\hspace{-3mm}
	\begin{minipage}{0.33\linewidth}
		\centerline{\includegraphics[width=4.5cm]{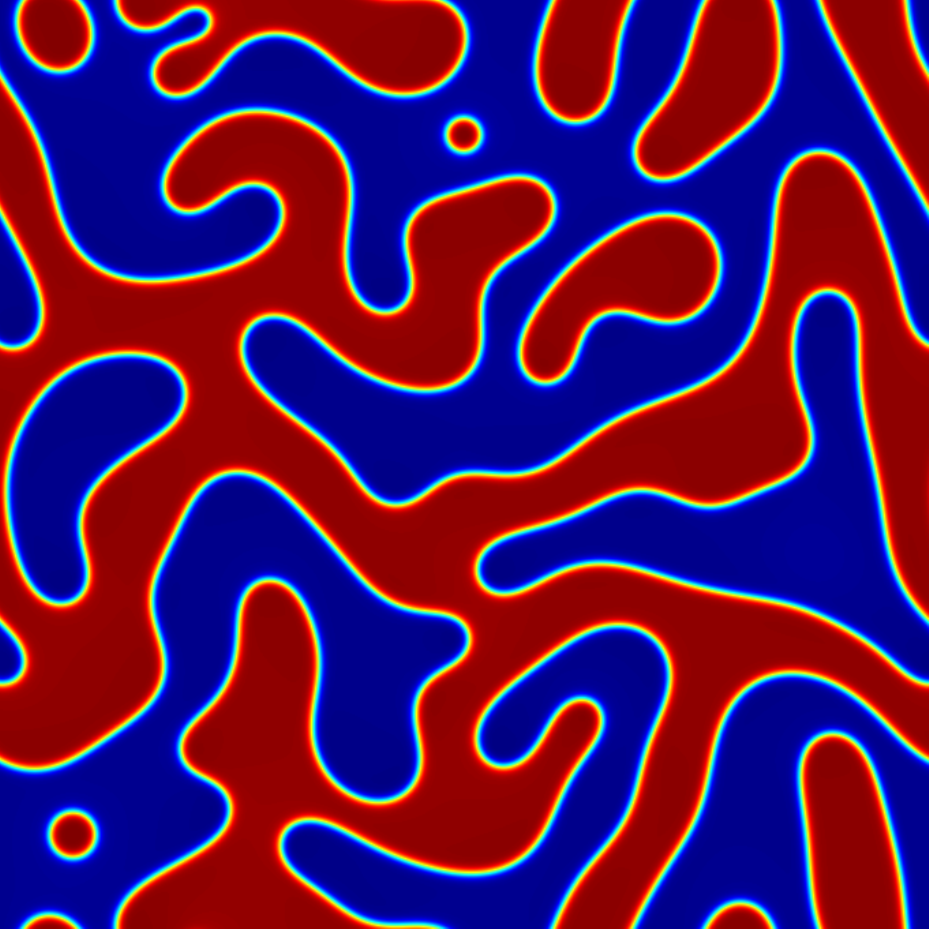}}
	\end{minipage}
	\hspace{-3mm}
	\begin{minipage}{0.33\linewidth}
		\centerline{\includegraphics[width=4.5cm]{6Kappa01.pdf}}
	\end{minipage}
	\vspace{8pt}
	\caption{The approximate solutions with $\kappa=0.1$ and $\epsilon=0.02,0.05,0.08$ at time $T = 150$.}
	\label{snapshotsepsilon}
\end{figure}

Evolutions of energy computed with different combinations of $\kappa$ and $\epsilon$ are presented in Figure \ref{comparekappa}. Figure \ref{comparekappa}(a) indicates that larger values of $\kappa$ lead to a faster energy decay rate for $\alpha=0.6$ and $\epsilon=0.08$. The results of the approximate solutions in Figure \ref{snapshotskappa} at time $T=150$ further confirm the roles of $\kappa$ in dynamical evolutions. 
Additionally, it is evident from Figure \ref{comparekappa}(b) that when $\alpha=0.6$ and $\kappa=0.1$, the time required for the solution to reach a stable state is affected by the parameter $\epsilon$.
What's more, we clearly observe that $\epsilon$ governs the interface thickness between two phases in Figure \ref{snapshotsepsilon}. The smaller $\epsilon$ is, the sharper the interface is.

\section*{Appendix. Proof of Lemma \ref{lem:DOC quadr form Ha-H1+a embedding inequ}}
We introduce auxiliary lemmas first before the proof of Lemma \ref{lem:DOC quadr form Ha-H1+a embedding inequ} here.  The following two $(n-1)\times (n-1)$ matrices are prepared
\[
B_2:=
\left(
\begin{array}{cccc}
	b_{0}^{(2)}  &             &            & \\
	b_{1}^{(3)} &b_{0}^{(3)}  &            & \\
	&\ddots       &\ddots      &\\
	&             &b_{1}^{(n)} &b_{0}^{(n)}  \\
\end{array}
\right)\quad\text{and}\quad
\Theta_2:=
\left(
\begin{array}{cccc}
	\theta_{0}^{(2)}  &                  &  & \\
	\theta_{1}^{(3)}  &\theta_{0}^{(3)}  &  & \\
	\vdots           &\vdots           &\ddots  &\\
	\theta_{n-2}^{(n)}&\theta_{n-3}^{(n)}&\cdots  &\theta_{0}^{(n)}  \\
\end{array}
\right),
\]
where the discrete kernels
$b_{n-k}^{(n)}$ and $\theta_{n-k}^{(n)}$ are defined by \eqref{def:BDF2-kernels} and \eqref{def: DOC-Kernels}, respectively.
It follows from the discrete orthogonal
identity  \eqref{eq: orthogonal identity} that $\Theta_2= B_2^{-1}.$
According to Lemma \ref{lem: DOC property} (I), the following symmetric matrix $
	\Theta:=\Theta_2+\Theta_2^T
	= B_2^{-1}+( B_2^{-1})^T
$
is also positive definite.
The estimate of the bound for the minimum and the maximum eigenvalues to $\Theta$ are given in the following. 

\begin{lemma}\cite[Lemma 3.4 $\&$ 3.5]{Liao2021math.NA}\label{lem: Theta minimum eigenvalue}
	If $0<r_k\le r_{\mathrm{user}}(<4.864)$ holds, then there exists positive constants $\mathfrak{m}_1$, $\mathfrak{m}_2$ and $\mathfrak{m}_3$ such that
	\begin{align*}
		\frac{\mathfrak{m}_1}{\mathfrak{m}_2}\,(\Lambda_\tau\myvec{v})^T(\Lambda_\tau\myvec{v})\le\,\myvec{v}^T\Theta\myvec{v}\le\,\mathfrak{m}_3\,(\Lambda_\tau\myvec{v})^T(\Lambda_\tau\myvec{v})\quad \text{for any vector $\myvec{v}$},
	\end{align*}
   where the the  diagonal matrix
   $\Lambda_\tau:=\text{diag}\bra{\sqrt{\tau_2},\sqrt{\tau_3},\cdots,\sqrt{\tau_n}}$.
\end{lemma}
The following lemma describe the Young-type convolution inequality. 
\begin{lemma}\cite[Lemma 3.7]{Liao2021math.NA}\label{lem:DOC quadr form Young inequ-embedding}
	If $0<r_k\le r_{\mathrm{user}}(<4.864)$ holds, then for any real sequences $\{v^k\}_{k=2}^n$ and $\{w^k\}_{k=2}^n$,
	\begin{align*}
		\sum_{k,j}^{n,k}\theta_{k-j}^{(k)} w^k v^j
		\le\varepsilon\sum_{k=2}^n\tau_k  (v^k)^2
		+\frac{\mathfrak{m}_3}{4\mathfrak{m}_1\varepsilon}
		\sum_{k=2}^{n}\tau_k (w^k)^2\quad \text{for $\forall\;\varepsilon > 0$}.
	\end{align*}
\end{lemma}

Now we are in the position to prove Lemma \ref{lem:DOC quadr form Ha-H1+a embedding inequ}.

\begin{proof}
	Assume $w^k\in \mathbb{V}_{h}$, taking $v^{j}:=u_{h}^{j} v_{h}^{j}$ and $\varepsilon:=\varepsilon_{1}$ in Lemma \ref{lem:DOC quadr form Young inequ-embedding}, then we have
	$$
	\sum_{k, j}^{n, k} \theta_{k-j}^{(k)}\myinnerb{ u^{j} v^{j}, w^{k}} \le \varepsilon_{1} \sum_{k=2}^{n} \tau_{k}\mynormb{u^{k} v^{k}}^{2}+\frac{\mathfrak{m}_{3}}{4 \mathfrak{m}_{1} \varepsilon_{1}} \sum_{k=2}^{n} \tau_{k}\mynormb{w^k}^{2} .
	$$
	The H$\rm\ddot{o}$lder inequality and the discrete embedding inequality \eqref{ieq: H1 embedding H1a} lead to
	$$\mynormb{u^{k} v^{k}} \le\mynormb{u^{k}}_{l^{3}}\mynormb{v^{k}}_{l^{6}}  \le c_{z} c_{u}\mynorm{\left(-\Delta_{h}\right)^{\frac{1+\alpha}{2}}v^k}^\frac{1}{1+\alpha}\mynormb{v^k}^\frac{\alpha}{1+\alpha},$$ which yields
	$$
	\sum_{k=2}^{n} \tau_{k}\mynormb{u^{k} v^{k}}^{2} \le c_{z}^{2} c_{u}^{2} \sum_{k=2}^{n} \tau_{k}\mynorm{\left(-\Delta_{h}\right)^{\frac{1+\alpha}{2}}v^k}^\frac{2}{1+\alpha}\mynormb{v^k}^\frac{2\alpha}{1+\alpha}.
	$$
	Then it follows that
	\begin{align}\label{3}
		\sum_{k, j}^{n, k} \theta_{k-j}^{(k)}\myinnerb{ u^{j} v^{j}, w^{k}} \le \varepsilon_{1} c_{z}^{2} c_{u}^{2} \sum_{k=2}^{n} \tau_{k}\mynorm{\left(-\Delta_{h}\right)^{\frac{1+\alpha}{2}}v^k}^\frac{2}{1+\alpha}\mynormb{v^k}^\frac{2\alpha}{1+\alpha}+\frac{\mathfrak{m}_{3}}{4 \mathfrak{m}_{1} \varepsilon_{1}} \sum_{k=2}^{n} \tau_{k}\mynormb{w^k}^{2} .	
	\end{align}		
	For the fixed time index $n$, we set $w^k:=-(-\Delta_h)^{\alpha} v^k$ and $\varepsilon_1:=\mathfrak{m}_2\mathfrak{m}_3/(2\mathfrak{m}_1^2\varepsilon_2)$ in \eqref{3},it gives
	\begin{align}\label{lemProof:DOC quadr form H2-H2 embedding inequ1}
		\sum_{k,j}^{n,k}\theta_{k-j}^{(k)} \myinnerb{u^jv^j,-(-\Delta_h)^{\alpha} v^k}
		\le\,&\frac{c_z^2c_u^2\mathfrak{m}_2\mathfrak{m}_3}{2\mathfrak{m}_1^2\varepsilon_2}
		\sum_{k=2}^n\tau_k\mynorm{\left(-\Delta_{h}\right)^{\frac{1+\alpha}{2}}v^k}^\frac{2}{1+\alpha}\left\|v^k\right\|^\frac{2\alpha}{1+\alpha}\nonumber\\
		&+\frac{\mathfrak{m}_1\varepsilon_2}{2\mathfrak{m}_2}\sum_{k=2}^{n}\tau_k \mynormb{(-\Delta_h)^\alpha v^k}^2.
	\end{align}
	For the first term of \eqref{lemProof:DOC quadr form H2-H2 embedding inequ1}, beginning with an application of Young's inequality, we arrive at
	\begin{align}\label{4}
		\sum_{k=2}^n\tau_k \mynormb{(-\Delta_h)^{\frac{1+\alpha}{2}} v^k}^{\frac{2}{1+\alpha}}
		\mynormb{v^k}^{\frac{2\alpha}{1+\alpha}}\nonumber
		\le&\, \frac{\varepsilon_3^{1+\alpha}}{1+\alpha}
		\sum_{k=2}^{n}\tau_k\mynormb{(-\Delta_h)^{\frac{1+\alpha}{2}}v^k}^2
		+\frac{\alpha}{(1+\alpha)\varepsilon_3^{\frac{1+\alpha}{\alpha}}}\sum_{k=2}^n\tau_k  \mynormb{v^k}^2\nonumber\\
		\le&\, \frac{2\mathfrak{m}_2\varepsilon_3^{1+\alpha}}{\mathfrak{m}_1(1+\alpha)}
		\sum_{k,j}^{n,k}\theta_{k-j}^{(k)}\myinnerb{(-\Delta_h)^{\frac{1+\alpha}{2}}v^j,(-\Delta_h)^{\frac{1+\alpha}{2}}v^k}\nonumber\\
		&+			\frac{\alpha}{(1+\alpha)\varepsilon_3^{\frac{1+\alpha}{\alpha}}}\sum_{k=2}^n\tau_k  \mynormb{v^k}^2.
	\end{align}
	The second term is handled as follows
	\begin{align}\label{5}
		\sum_{k=2}^{n}\tau_k \mynormb{(-\Delta_h)^\alpha v^k}^2\le&\,\sum_{k=2}^{n}\tau_k \mynormb{(-\Delta_h)^{\frac{1+\alpha}{2} }v^k}^2\nonumber\\
		\le&\, \frac{2\mathfrak{m}_2}{\mathfrak{m}_1}\sum_{k,j}^{n,k}\theta_{k-j}^{(k)}\myinnerb{(-\Delta_h)^{\frac{1+\alpha}{2}} v^j,(-\Delta_h)^{\frac{1+\alpha}{2}} v^k},		
	\end{align}
	where Lemma \ref{lem: Theta minimum eigenvalue} is used in the last step.
	Substituting inequality \eqref{4}-\eqref{5} into \eqref{lemProof:DOC quadr form H2-H2 embedding inequ1}, we have the following estimate
	\begin{align*}
		\sum_{k,j}^{n,k}\theta_{k-j}^{(k)} \myinnerb{u^jv^j,\Delta_h v^k}
		\le&\,\braB{\frac{c_z^2c_u^2\mathfrak{m}_2^2\mathfrak{m}_3\varepsilon_3^{1+\alpha}}{\mathfrak{m}_1^3(1+\alpha)\varepsilon_2}+\varepsilon_2}
		\sum_{k,j}^{n,k}\theta_{k-j}^{(k)}\myinnerb{(-\Delta_h)^{\frac{1+\alpha}{2}}v^j,(-\Delta_h)^{\frac{1+\alpha}{2}}v^k}\\
		&\,+\frac{c_z^2c_u^2\mathfrak{m}_2\mathfrak{m}_3\alpha}{2\varepsilon_2\mathfrak{m}_1^2(1+\alpha)\varepsilon_3^{\frac{1+\alpha}{\alpha}}}\sum_{k=2}^n\tau_k\mynormb{v^k}^2.
	\end{align*}
	Now by choosing $\varepsilon_2:=\varepsilon/2$ and $\varepsilon_3^{1+\alpha}:=\varepsilon^2\mathfrak{m}_1^3(1+\alpha)/(4c_z^2c_u^2\mathfrak{m}_2^2\mathfrak{m}_3)$,
	we obtain the claimed inequality.
\end{proof}
\section*{Acknowledgement}
\noindent We would like to acknowledge support by the National Natural Science Foundation of China (No. 11701081,11861060), the State Key Program of National Natural Science Foundation of China (61833005), ZhiShan Youth Scholar Program of SEU.

\end{document}